\documentclass[11pt,reqno]{amsart}
\pdfoutput=1
\usepackage[utf8]{inputenc}
\usepackage[
 a4paper,
 total={160mm,257mm},
 left=25mm,
 top=20mm]{geometry}
\usepackage{hyperref}
\usepackage{amsfonts}
\usepackage{amsmath}
\usepackage{amssymb}
\usepackage{mathrsfs}  
\usepackage{color}
\usepackage{physics}
\usepackage{multicol}
\usepackage{setspace}
\usepackage{amsthm}
\usepackage{caption}
\usepackage{subcaption}

\usepackage{tikz}
\usetikzlibrary{knots}
\usetikzlibrary{calc}
\usetikzlibrary{decorations.pathreplacing} 
\usetikzlibrary{decorations.markings}
\usetikzlibrary{arrows}
\usetikzlibrary{arrows.meta}
\tikzset{->-/.style={decoration={
  markings,
  mark=at position .5 with {\arrow{>}}},postaction={decorate}}}

\hypersetup{colorlinks,breaklinks,
            linkcolor = purple,
            citecolor = teal,
            urlcolor = purple,
            }

\theoremstyle{definition}
\newtheorem{theorem}{Theorem}[section]

\newtheorem{corollary}[theorem]{Corollary}

\newtheorem{lemma}[theorem]{Lemma}
\newtheorem{proposition}[theorem]{Proposition}
\theoremstyle{remark}

\newtheorem{remark}[theorem]{Remark}

\newcommand{\ee}{\mathrm{e}}
\newcommand{\ii}{\mathrm{i}}

\newcommand{\rmT}{\scalebox{0.5}{$\mathrm T$}}

\numberwithin{equation}{section}

\title[$q$-deformation of random partitions]{$q$-deformation of random partitions,\\ determinantal structure,\\ and Riemann--Hilbert problem} 
\author[Taro Kimura]{Taro Kimura}
\address{Université Bourgogne Europe, CNRS, IMB UMR 5584, Dijon, France}

\begin{document}

\begin{abstract}
We study $q$-deformation of probability measures on partitions, i.e., $q$-deformed random partitions.
We in particular consider the $q$-Plancherel measure and show a determinantal formula for the correlation function using a $q$-deformation of the discrete Bessel kernel.
We also investigate Riemann-Hilbert problems associated with the corresponding orthogonal polynomials and obtain $q$-Painlevé equations from the $q$-difference Lax formalism.
\end{abstract}

\maketitle

\tableofcontents

\section{Introduction and summary}

A partition of a non-negative integer is a ubiquitous object appearing in various combinatorial aspects of mathematics and physics.
Our main interest is a probability distribution of partitions, i.e., random partitions, and its large scale behavior.
The Plancherel measure is a primary example of the probability measure on the set of partitions of size $n$ through the identification of each partition as the irreducible representation of the symmetric group $\mathfrak{S}_n$,
\begin{align}
    \mu_\text{P}[\lambda] = \frac{(\dim \lambda)^2}{n!} \, , \qquad \lambda \vdash n \, ,
\end{align}
where we denote by $\dim \lambda$ the dimension of the irreducible representation of $\mathfrak{S}_n$ labeled by $\lambda$.
We are also interested in the Poissonized Plancherel measure on the set of partitions of arbitrary size, which is a one-parameter extension given by
\begin{align}
    \mu_\text{PP}[\lambda;\eta] = \ee^{-\eta^2} \eta^{2|\lambda|} \left( \frac{\dim \lambda}{|\lambda|!} \right)^2 = \ee^{-\eta^2} \eta^{2|\lambda|} \prod_{(i,j) \in \lambda} \frac{1}{h(i,j)^2} \, ,
\end{align}
where $h(i,j)$ is the hook length associated with the position $(i,j)$ and we denote the size of $\lambda$ by $|\lambda|$.
In this case, there is no restriction on the size of partitions.
For such a probability measure, we are typically interested in the large scale behavior and its universality.
The result presented by Logan and Shepp~\cite{Logan:1977}, and Vershik and Kerov~\cite{Vershik:1977} about the so-called limit shape of random partitions is the pioneering work in this direction.
Moreover, Baik, Deift, and Johansson showed in their seminal work~\cite{Baik:1999JAMS} that the large scale behavior of random permutations is described by the Tracy--Widom distribution of Gaussian Unitary Ensemble (GUE)~\cite{Tracy:1992rf}, which establishes the universality of random partitions in the large scale limit.
Since then, there have been numerous works on stochastic and probabilistic aspects of random partitions in various contexts.

In this paper, we study $q$-deformation of the Poissonized Plancherel measure on the set of partitions $\mathsf{P}$, that we call the $q$-Plancherel measure.
Let $q, \xi \in [0,1)$.
We define two types of the $q$-Plancherel measure,
\begin{subequations}
    \begin{align}
        \text{Squared type:} \quad & \mu_{q\text{PP}_S}[\lambda;\xi,q] = M(\xi;q)^{-1} (\xi^2 q)^{|\lambda|} q^{2b(\lambda)} \prod_{(i,j) \in \lambda} \frac{1}{(1 - q^{h(i,j)})^2} \, , \\
        \text{Mixed type:} \quad & \mu_{q\text{PP}_M}[\lambda;\xi,q] = \exp\left( \frac{\xi^2}{q^{\frac{1}{2}} - q^{-\frac{1}{2}}} \right) \xi^{2|\lambda|} q^{b(\lambda)} \prod_{(i,j) \in \lambda} \frac{1}{h(i,j)(1 - q^{h(i,j)})} \, ,
    \end{align}
\end{subequations}
for $\lambda \in \mathsf{P}$, where $b(\lambda)$ is defined in \eqref{eq:b-partition}, and $M(\xi;q)$ is the modified MacMahon function~\eqref{eq:MacMahon_fn}.
See \S\ref{sec:qP_measure} for more details.
The squared type was introduced by Fulman~\cite{Fulman1997,Fulman2002}, and the mixed type was introduced by Strahov~\cite{Strahov2008}. 

We first point out that both measures are in fact specialization of the multi-variable probability measure on partitions, a.k.a., the Schur measure, introduced by Okounkov~\cite{Okounkov:2001SM} (Proposition~\ref{prop:Schur_to_qPlancherel}).
This identification allows us to show that, in contrast to other $q$-deformation of random partitions, e.g., Macdonald measure, $q$-Whittaker measure~\cite{Borodin:2014PTRF}, the $q$-Plancherel measure of both types are discrete determinantal point processes (Corollary~\ref{cor:dDDP}): All the correlation functions are given by a determinant of the correlation kernel.
In particular, for the squared type measure, the correlation kernel is given as follows.
\begin{theorem}[Theorem~\ref{thm:q-Bessel_kernel}]
    The correlation kernel of the squared type $q$-Plancherel measure is given by the $q$-Bessel kernel,
    \begin{align}
        K_{q\text{B}}(r,s) = \xi \frac{J_{r+\frac{1}{2}}J_{s-\frac{1}{2}} - J_{r-\frac{1}{2}}J_{s+\frac{1}{2}}}{q^{\frac{1}{2}(r-s)} - q^{-\frac{1}{2}(r-s)}} \, , \qquad r,s \in \mathbb{Z} + \frac{1}{2} \, ,
    \end{align}
    where $J_n = J^{(3)}_n(2\xi;q)$ is the Hahn--Exton $q$-Bessel function.
\end{theorem}
See \S\ref{sec:q-Bessel} for the details of the Hahn--Exton $q$-Bessel function.
This is a natural $q$-deformation of the well-known discrete Bessel kernel introduced by Borodin, Okounkov, and Olshanski~\cite{Brodin:2000} and Johansson~\cite{Johansson2001}.
We verify that the $q$-Bessel kernel is reduced to the discrete Bessel kernel by putting $\xi = (1 - q)\eta$ and taking the limit $q \to 1$. 
Similar to the discrete Bessel kernel, we can take the scaling limit of the $q$-Bessel kernel to obtain the universal kernels, i.e., the sine kernel and the Airy kernel.
We determine the scaling limit to realize these universal kernels.
From this analysis, we also obtain the limit shape for the $q$-Plancherel measure (Fig.~\ref{fig:prof_fn}).

We then study the gap probability for the $q$-Plancherel measure.
Since the $q$-Plancherel measure is determinantal, the gap probability is given by a discrete version of the Fredholm determinant.
As shown by Borodin and Okounkov~\cite{Borodin2000}, the discrete Fredholm determinant associated with the specific interval is presented by the Toeplitz determinant.
For the squared type measure, we have the following Toeplitz determinant formula.
\begin{proposition}[Propositions~\ref{prop:gap_prob_unitary_matrix}, \ref{prop:Toeplitz_Z}]
    The gap probability with respect to the squared type $q$-Plancherel measure is given by the Toeplitz determinant,
    \begin{align}
        \mathbb{P}_{q\text{PP}_S}[\ell(\lambda) \le N] = \frac{1}{M(\xi;q)} \det_{0 \le i, j \le N - 1} I_{-i + j}
        \, , \quad 
        \mathbb{P}_{q\text{PP}_S}[\lambda_1 \le N] = \frac{1}{M(\xi;q)} \det_{0 \le i, j \le N - 1} \check{I}_{-i + j} \, ,
    \end{align}
    where $I_n = I_n^{(1)}(2 \xi q^{\frac{1}{2}};q)$ and $\check{I}_n = q^{\frac{n^2}{2}} I_n^{(2)}(2\xi;q)$ are the modified $q$-Bessel functions, and $M(\xi;q)$ is the modified MacMahon function.
\end{proposition}
See \S\ref{sec:q-functions} for the details of the modified MacMahon function and \S\ref{sec:q-Bessel} for the modified $q$-Bessel function.
We remark that, for the $q$-Plancherel measure, the measure is not invariant under the transposition.
Hence, we have two different expressions for the gap probability with respect to $\lambda_1$ and $\ell(\lambda) = \lambda_1^{\rmT}$.

In order to derive the Toeplitz determinant formula, we use the unitary matrix integral and the orthogonal polynomials.
The inner product for the orthogonal polynomials appearing in this context is defined on the integral on the unit circle, which shows an analogous structure to the $q$-orthogonal polynomials as we will see in \S\ref{sec:q-orthogonal_polynomial} (see, e.g., \cite{Ismail2005,Koekoek2010}).
We study the Riemann--Hilbert problem associated with these orthogonal polynomials and also the corresponding linear $q$-difference system based on the Lax formalism:
For a matrix-valued function $\Psi$ constructed from the orthogonal polynomials, we have
\begin{align}
    \Psi_{n+1}(z) = U_n(z) \Psi_n(z) \, , \qquad \Psi_n(qz) = T_n(z) \Psi_n(z) \, .
\end{align}
We call $(U, T)$ the Lax pair associated with the linear $q$-difference system.
Since we have $\Psi_{n+1}(qz) = U_n(qz) T_n(z) \Psi_n(z) = T_{n+1}(z) U_n(z) \Psi_n(z)$, and $\Psi$ is invertible on the complex plane except for a certain set, we have the so-called compatibility condition,
\begin{align}
    U_n(qz) T_n(z) = T_{n+1}(z) U_n(z) \, ,
\end{align}
from which we deduce the following.
\begin{theorem}[Theorems~\ref{thm:qPV1}, \ref{thm:qPV2}]
    For the monic orthogonal polynomials associated with the $q$-Plancherel measure, $\{\pi_n\}_{n \in \mathbb{Z}_{\ge 0}}$ and $\{\check{\pi}_n\}_{n \in \mathbb{Z}_{\ge 0}}$, set $\mathsf{x}_n = \xi^{\frac{1}{2}} q^{\frac{n}{2}} \pi_n(0)$ and $\mathsf{y}_n = (-\xi)^{\frac{1}{2}} q^{-\frac{n}{2}} \check{\pi}_n(0)$.
    The following recurrence relations hold,
    \begin{subequations}
    \begin{align}
        \left( \mathsf{x}_n \mathsf{x}_{n+1} - 1 \right) \left( \mathsf{x}_{n-1} \mathsf{x}_{n} - 1 \right)  & = \frac{(\mathsf{x}_n^2 - \xi)(\mathsf{x}_n^2 - \xi^{-1})}{(1 - \xi^{-1} q^{-n} \mathsf{x}_n^2)} \, , \\
        \left( \mathsf{y}_n \mathsf{y}_{n+1} - 1 \right) \left( \mathsf{y}_{n-1} \mathsf{y}_{n} - 1 \right)  & = \frac{(\mathsf{y}_n^2 + \xi)(\mathsf{y}_n^2 + \xi^{-1})}{(1 + \xi^{-1} q^{-n} \mathsf{y}_n^2)} \, ,
    \end{align}        
    \end{subequations}
    where $\mathsf{x}_0=\xi^{\frac{1}{2}}$ and $\mathsf{y}_0=(-\xi)^{\frac{1}{2}}$.
\end{theorem}
These equations are special cases of the $q$-Painlevé V equation ($q$-P$_\text{V}$).%
\footnote{
Historically, $q$-P$_\text{V}$ had been initially introduced as the discrete Painlevé V (d-P$_\text{V}$)
in~\cite{Ramani:1991zz}.
Afterwards, an appropriate version of d-P$_\text{V}$ was then found, and the equation found in \cite{Ramani:1991zz} is now understood as $q$-P$_\text{V}$. 
See, e.g.,~\cite[\S4]{Grammaticos2004} for details.
}
We have a shifted Toeplitz determinant formula for $\mathsf{x}_n$ and $\mathsf{y}_n$,
\begin{align}
        \mathsf{x}_n = (-q^{\frac{1}{2}})^n \xi^{\frac{1}{2}} \frac{Z_n^{(1)}}{Z_n^{(0)}} \, , \qquad
        \mathsf{y}_n = (-q^{-\frac{1}{2}})^n (-\xi)^{\frac{1}{2}} \frac{\check{Z}_n^{(1)}}{\check{Z}_n^{(0)}} \, ,
\end{align}
where
\begin{align}
    Z_n^{(k)} = \det_{0 \le i, j \le n-1} I_{-i+j-k} \, , \qquad
    \check{Z}_n^{(k)} = \det_{0 \le i, j \le n-1} \check{I}_{-i+j-k} \, .
\end{align}
We remark that a similar determinant solution using the $q$-Bessel function to the $q$-Painlevé III equation ($q$-P$_\text{III}$) was obtained in~\cite{Kajiwara:1995JMP}.
They considered a Casorati type determinant with respect to the $q$-shift of the parameter $\xi$ instead of the index $n$ as we used.
See also a Toeplitz determinant formula for the symmetric form of the $q$-Painlevé IV equation ($q$-P$_\text{IV}$)~\cite{Kajiwara2000}.

Another remark is that the $q$-P$_\text{V}$ was obtained also from other kinds of $q$-orthogonal polynomials for the $q$-Freund weight by Boelen--Smet--Van Assche~\cite{Boelen2010}, for the semi-classical (little) $q$-Laguerre weight by Filipuk--Smet~\cite{Filipuk2015} and Boelen--Van Assche~\cite{Boelen2013}, and for the $q$-Hahn weight by Knizel~\cite{Knizel2016}.
See also a monograph on the connection between the Painlevé equations and orthogonal polynomials~\cite{VanAssche2017}.

We verify that the $q$-Painlevé equation above is reduced to the discrete Painlevé II equation (d-P$_\text{II}$) in the limit $q \to 1$ under the parametrization $\xi = (1 - q)\eta$ as before.
This behavior is consistent with the well-known result for the relation between the ordinary Plancherel measure to the discrete Painlevé equation by Baik~\cite{Baik2003}, Borodin~\cite{Borodin:2003Duke}, and Adler and van Moerbeke~\cite{Adler:2003CMP}.

\subsubsection*{Notations}

We denote by $\mathsf{P}$ a set of partitions,  $\lambda = (\lambda_1 \ge \lambda_2 \ge \cdots \ge \lambda_\ell > 0)$, where $\ell$ is the length of $\lambda$ denoted by $\ell(\lambda)$.
We denote by $\mathsf{P}_N$ the set of partitions whose length is less than $N$, $\mathsf{P}_N = \{ \lambda \in \mathsf{P} \mid \ell(\lambda) \le N\}$, i.e., $\mathsf{P} = \lim_{N \to \infty} \mathsf{P}_N$.
We denote the transposition of $\lambda$ by $\lambda^{\rmT}$.
Then, we have $\ell(\lambda) = \lambda_1^{\rmT}$.
The size of a partition $\lambda$ is denoted by $|\lambda| = \sum_{i=1}^{\ell(\lambda)} \lambda_i$, i.e., $\lambda \vdash n \iff |\lambda| = n$.
Note that $\mathsf{P}_N \neq \{\lambda \in \mathsf{P} \mid |\lambda| = N \}$.
For a partition $\lambda$, we define
\begin{align}
    b(\lambda) = \sum_{i=1}^{\ell(\lambda)} (i-1) \lambda_i \, . \label{eq:b-partition}
\end{align}
For $(i,j) \in \lambda$, we define a hook length,
\begin{align}
    h(i,j) = \lambda_i + \lambda_j^{\rmT} - i - j + 1 \, .
\end{align}

\subsubsection*{Acknowledgment}

This work has received financial support from the CNRS through the MITI interdisciplinary programs, EIPHI Graduate School (No. ANR-17-EURE-0002) and Bourgogne-Franche-Comté region.
We are grateful to the anonymous referees for their constructive comments and suggestions.

\section{$q$-Plancherel measure}\label{sec:qP_measure}

The Plancherel measure is a probability measure on the set of irreducible unitary representations of a locally compact group $G$.
For the symmetric group $G = \mathfrak{S}_n$, irreducible representations are parametrized by partitions $\lambda \vdash n$, which allows us to interpret the corresponding Plancherel measure as a probability measure on the set of partitions,
\begin{align}
    \mu_\text{P}[\lambda] = \frac{(\dim \lambda)^2}{n!} \, , \qquad \lambda \vdash n \, ,
\end{align}
where we denote by $\dim \lambda$ the dimension of the irreducible representation characterized by $\lambda$, which agrees with the number of standard Young tableaux of shape $\lambda$.
There is a combinatorial formula to compute the dimension, which is called the hook length formula,
\begin{align}
    \frac{\dim \lambda}{n!} = \prod_{(i,j) \in \lambda} \frac{1}{h(i,j)} \, . \label{eq:hook_length_formula}
\end{align}
Then, the normalization condition of the Plancherel measure is ensured by a combinatorial identity obtained via the Robinson--Schensted correspondence between permutations and pairs of standard Young tableaux associated with the same shape,
\begin{align}
    \sum_{\lambda \vdash n} (\dim \lambda)^2 = n! \, .
\end{align}
This is called the Frobenius--Young identity.

The Poissonized Plancherel measure is a one-parameter deformation of the Plancherel measure obtained via Poissonization and defined on the whole set of partitions,
\begin{align}
    \mu_\text{PP}[\lambda] = \mu_\text{PP}[\lambda;\eta] = \ee^{-\eta^2} \eta^{2|\lambda|} \left( \frac{\dim \lambda}{|\lambda|!} \right)^2 \, , \qquad \lambda \in \mathsf{P} \, .
    \label{eq:PP_measure}
\end{align}

Let us introduce two types of $q$-deformation of the (Poissonized-)Plancherel measure.
Set $q, \xi \in [0,1)$ in the following.%
\footnote{%
Most of the following discussions are based on algebraic computations, which still holds for $q, \xi \in \mathbb{C}^\times$ under the condition $|q|, |\xi| < 1$.
}

\subsubsection*{Squared type}
The first $q$-deformation is given as follows,
\begin{align}
    \mu_{q\text{PP}_S}[\lambda] = \mu_{q\text{PP}_S}[\lambda;\xi,q] = M(\xi;q)^{-1} (\xi^2 q)^{|\lambda|} q^{2b(\lambda)} \prod_{(i,j) \in \lambda} \frac{1}{(1 - q^{h(i,j)})^2} \, ,
\end{align}
where $M(\xi;q)$ is the modified MacMahon function~\eqref{eq:MacMahon_fn}.
We call this the squared type $q$-Plancherel measure.
This measure was introduced by Fulman~\cite[\S2.9]{Fulman1997} as a specialization of the measure based on the Macdonald polynomial, which is a further deformation of the Schur measure presented in \S\ref{sec:Schur_measure}.
As in the case of the Schur polynomial, the Macdonald polynomial has the corresponding the Cauchy identity, from which one may define a measure based on the Macdonald polynomial.

This measure has been also studied in the context of supersymmetric gauge theory and topological string theory~\cite{Nekrasov:2002qd,Nekrasov:2003rj,Okounkov:2003sp}.
In fact, the modified MacMahon function has a two-fold interpretation as the partition function of five-dimensional $\mathcal{N}=1$ supersymmetric $\mathrm{U}(1)$ gauge theory on $\mathbb{C}^2 \times \mathbb{S}^1$ and the topological string partition function of the toric Calabi--Yau 3-fold, $X = \operatorname{Tot}_{\mathbb{P}^1}(\mathcal{O}(0)\oplus\mathcal{O}(-2))$. 

\subsubsection*{Mixed type}
The second $q$-deformation is given as follows,
\begin{align}
    \mu_{q\text{PP}_M}[\lambda] = \mu_{q\text{PP}_M}[\lambda;\xi,q] = \exp\left( \frac{\xi^2}{q^{\frac{1}{2}} - q^{-\frac{1}{2}}} \right) \xi^{2|\lambda|} q^{b(\lambda)} \prod_{(i,j) \in \lambda} \frac{1}{h(i,j)(1 - q^{h(i,j)})} \, .
\end{align}
We call this the mixed type $q$-Plancherel measure.
This measure, a special case of the so-called knot measure~\cite[Chapter~3, \S4]{Kerov2003} (see also an earlier work~\cite{Kerov1993}), was introduced by Strahov~\cite[\S2.2]{Strahov2008}.
In fact, Strahov considered the de-Poissonized version of this measure, from the point of view of the Iwahori--Hecke algebra, which is a $q$-deformation of the symmetric group.
He also showed that this measure arises from the non-uniform random permutations through the Robinson--Schensted--Knuth correspondence.

\begin{proposition}
    We have
\begin{align}
    \lim_{q \to 1} \mu_{q\text{PP}_S}[\lambda;(1-q)\eta;q] = \lim_{q \to 1} \mu_{q\text{PP}_M}[\lambda;(1-q)^{\frac{1}{2}} \eta;q]  = \mu_\text{PP}[\lambda;\eta] \, .
\end{align}
\end{proposition}

\begin{proof}
Let us first consider the squared type measure $\mu_{q\text{PP}_S}$. 
By Proposition~\ref{prop:MacMahon_PE}, we have
\begin{align}
    M((1-q)\eta,q)^{-1} = \exp \left( - \sum_{n \ge 1} \frac{(1-q)^{2n} \eta^{2n}}{n(q^{\frac{n}{2}} - q^{-\frac{n}{2}})^2} \right) \xrightarrow{q \to 1} \ee^{-\eta^2} \, ,
\end{align}
which agrees with the normalization of $\mu_\text{PP}$.
Then, we have
\begin{align}
    \left( (1-q) \eta \right)^{2|\lambda|} \prod_{(i,j) \in \lambda} \frac{1}{(1-q^{h(i,j)})^2} = \eta^{2|\lambda|} \prod_{(i,j) \in \lambda} \left( \frac{1-q}{1-q^{h(i,j)}} \right)^2 \xrightarrow{q \to 1} \eta^{2|\lambda|} \prod_{(i,j) \in \lambda} \frac{1}{h(i,j)^2} \, .
\end{align}
By the hook length formula~\eqref{eq:hook_length_formula}, this agrees with the $\lambda$-dependent part of $\mu_\text{PP}$.
The remaining $q$-factors become trivial in the limit, $\lim_{q \to 1} q^{|\lambda|} q^{2b(\lambda)} = 1$.
This completes the proof for the squared type measure.

For the mixed type measure $\mu_{q\text{PP}_M}$, we have
\begin{align}
    \left.\exp \left( \frac{\xi^2}{q^{\frac{1}{2}} - q^{-\frac{1}{2}}} \right)\right|_{\xi = (1-q)^{\frac{1}{2}}\eta} = \exp \left( \frac{(1-q)\eta^2}{q^{\frac{1}{2}} - q^{-\frac{1}{2}}} \right) \xrightarrow{q \to 1} \ee^{-\eta^2} \, ,
\end{align}
which agrees with the normalization of $\mu_\text{PP}$.
The $\lambda$-dependent part is then given by
\begin{align}
    \left( (1-q)^{\frac{1}{2}} \eta \right)^{2|\lambda|} \prod_{(i,j) \in \lambda} \frac{1}{h(i,j)(1 - q^{h(i,j)})} = \eta^{2|\lambda|} \prod_{(i,j) \in \lambda} \frac{1-q}{h(i,j)(1 - q^{h(i,j)})} \xrightarrow{q \to 1} \eta^{2|\lambda|} \prod_{(i,j) \in \lambda} \frac{1}{h(i,j)^2} \, .
\end{align}
The remaining $q$-factors become trivial in the limit, $\lim_{q \to 1} q^{b(\lambda)} = 1$, and hence we conclude the proof.
\end{proof}

\subsection{Reduction from Schur measure}\label{sec:Schur_measure}

The Schur measure is a multi-variable probability measure on partitions introduced by Okounkov~\cite{Okounkov:2001SM}.
We first summarize the Schur measure and show its relation to the $q$-Plancherel measure.

\subsubsection*{Schur function}

Let $s_\lambda(X)$ be a Schur function of (infinitely many) variables $X = (x_1,x_2,\ldots)$ associated with a partition $\lambda$.
We also use the notation,
\begin{align}
    s_\lambda(\{ t_n \}) = s_\lambda(X) 
    \quad \text{where} \quad t_n = \frac{1}{n} \sum_{i \ge 1} x_i^n \, .
\end{align}
The Cauchy identities for the Schur functions read
\begin{subequations}\label{eq:Cauchy_ids}
    \begin{align}
        \sum_{\lambda \in \mathsf{P}} s_\lambda(X) s_\lambda(Y) & = \prod_{i,j \ge 1} (1 - x_i y_j)^{-1} = \exp \left( \sum_{n \ge 1} \frac{1}{n} \tr X^n \tr Y^n \right) \, , \\
        \sum_{\lambda \in \mathsf{P}} s_{\lambda^{\rmT}}(X) s_\lambda(Y) & = \prod_{i,j \ge 1} (1 + x_i y_j) = \exp \left( - \sum_{n \ge 1} \frac{(-1)^n}{n} \tr X^n \tr Y^n \right) \, ,
    \end{align}
\end{subequations}
where we write $\tr X^n = \sum_{i \ge 1} x_i^n$ and $\tr Y^n = \sum_{i \ge 1} y_i^n$.

Set $\mathbb{T}_N=\left\{ (z_1,\ldots,z_N) \in (\mathbb{C}^\times)^N \mid |z_i|=1, i = 1, \ldots, N \right\}$.
The Schur function of $N$-variable (Schur polynomial) has the following orthonormal property,
\begin{align}
    \frac{1}{N!} \oint_{\mathbb{T}_N} s_\lambda(z_1,\ldots,z_N) \overline{s_{\lambda'}(z_1,\ldots,z_N)} \prod_{1 \le i < j \le N} |z_i - z_j|^2 \prod_{i=1}^N \frac{\dd{z}_i}{2 \pi \ii z_i} = \delta_{\lambda,\lambda'} \, .
    \label{eq:Schur_orthonormal}
\end{align}
We remark that $s_\lambda(z_1, \ldots,z_N) = 0$ if $\ell(\lambda) > N$.

\subsubsection*{Schur measure}
Let $x_i, y_i \in [0,1)$ for $i \in \mathbb{N}$.
The Schur measure on the set of partitions is defined using the Schur function,
\begin{align}
    \mu_\text{S}[\lambda] = \mu_\text{S}[\lambda;X,Y] = \frac{1}{Z(X,Y)} s_\lambda(X) s_\lambda(Y) \, , 
\end{align}
where the normalization factor is given by
\begin{align}
    Z(X,Y) = \sum_{\lambda \in \mathsf{P}} s_\lambda(X) s_\lambda(Y) = \prod_{i, j \ge 1} (1 - x_i y_j)^{-1} \, .
\end{align}
The parameters $(x_i,y_i)_{i \in \mathbb{N}}$ can be in general complex: We need a sufficient condition $X = \overline{Y}$ so that $\mu_\text{S}[\lambda] \ge 0$ for any $\lambda$.
We also use the notation,
\begin{align}
    \mu_\text{S}[\lambda;t,\tilde{t}] = \mu_\text{S}[\lambda;X,Y]
    \qquad \text{where} \qquad
    t_n = \frac{1}{n} \sum_{i \ge 1} x_i^n \, , \quad 
    \tilde{t}_n = \frac{1}{n} \sum_{i \ge 1} y_i^n \, .
\end{align}

\subsubsection*{Correlation function}
Let us consider the correlation function for the Schur measure.
For a partition $\lambda \in \mathsf{P}$, we define the boson-fermion map,
\begin{align}
    \Xi(\lambda) = \left\{ \lambda_i - i + \frac{1}{2} \right\}_{i \in \mathbb{N}} \subset \mathbb{Z} + \frac{1}{2} =: \mathbb{Z}' \, .
\end{align}
For a finite subset $\mathsf{Z} \subset \mathbb{Z}'$, we define the correlation function,
\begin{align}
    \rho(\mathsf{Z}) = \mathbb{P}_\text{S}[\mathsf{Z} \subset \Xi(\lambda)] \, ,
\end{align}
where we denote the probability with respect to the Schur measure by $\mathbb{P}_\text{S}[\cdot]$.
When $|\mathsf{Z}| = k$, it is called the $k$-point correlation function.

\begin{theorem}[Okounkov~\cite{Okounkov:2001SM}]
    For $z_1,\ldots,z_k \in \mathbb{Z}'$, the $k$-point correlation function of the Schur measure is given by a determinant of the form,
    \begin{align}
        \rho(z_1,\ldots,z_k) = \det_{1 \le i, j \le k} K(z_i,z_j) \, .
    \end{align}
    The correlation kernel is given by
    \begin{align}
        K(r,s) = \frac{1}{(2 \pi \ii)^2} \oint_{|z| > |w|} \frac{\mathbb{J}(z;t,\tilde{t}) \mathbb{J}(w;\tilde{t},t)}{z - w} \frac{\dd{z} \dd{w}}{z^{r+\frac{1}{2}} w^{-s+\frac{1}{2}}}
        \, , \label{eq:Schur_corr_kernel_int}
    \end{align}
    where the integration contour counterclockwise encircles the origin keeping the relation $|z| > |w|$ with the auxiliary function,
    \begin{align}
        \mathbb{J}(z;t,\tilde{t}) = \exp \left( \sum_{n \ge 1} ( t_n z^n - \tilde{t}_n z^{-n} ) \right) \, . \label{eq:J_aux_fn}
    \end{align}
    Therefore, the Schur measure is a discrete determinantal point process.
\end{theorem}

\begin{proposition}\label{prop:Schur_kernel_series}
    The correlation kernel of the Schur measure has the series expansion,
    \begin{align}
        K(r,s) = \sum_{k \in \mathbb{Z}_{>0}} J_{r+k} \tilde{J}_{s+k} \, ,
    \end{align}
    with 
    \begin{align}
        J_n = \frac{1}{2 \pi \ii} \oint \mathbb{J}(z;t,\tilde{t}) \frac{\dd{z}}{z^{n+1}} \, , \qquad 
        \tilde{J}_n = \frac{1}{2 \pi \ii} \oint \mathbb{J}(z;\tilde{t},t) \frac{\dd{z}}{z^{n+1}} \, ,
    \end{align}
    where the integration contour counterclockwise encircles the origin. 
\end{proposition}

The following specialization is technically important in the rest of the paper.
\begin{proposition}\label{prop:Schur_to_qPlancherel}
    The $q$-Plancherel measures of both types are obtained from the Schur measure under the specialization,
    \begin{subequations}
        \begin{align}
            \text{Squared type} : \quad &
            t_n = \tilde{t}_n = - \frac{\xi^n}{n(q^{\frac{n}{2}}-q^{-\frac{n}{2}})} \, , \label{eq:sq_specialization} \\
            \text{Mixed type} : \quad &
            t_n = - \frac{\xi^n}{n(q^{\frac{n}{2}}-q^{-\frac{n}{2}})} \, , \quad \tilde{t}_n = \xi q^{-\frac{1}{2}} \delta_{n,1} \, .
        \end{align}
    \end{subequations}
\end{proposition}

\begin{proof}
    We may use the following identity of the Schur polynomial (e.g.,~\cite[{\S}I.3]{Macdonald:1997}),
    \begin{align}
        s_\lambda(1,q,\ldots,q^{N-1}) = q^{b(\lambda)} \prod_{(i,j) \in \lambda} \frac{1-q^{N+c(i,j)}}{1-h(i,j)} \, ,\quad c(i,j) = j - i \, .
    \end{align}
    Taking the limit $N \to \infty$ together with the homogeneous property, $s_\lambda(a z_1,\ldots,a z_N) = a^{|\lambda|} s_\lambda(z_1,\ldots,z_N)$, we obtain the identities
    \begin{subequations}
        \begin{align}
        s_\lambda\left(\left\{t_n = - \frac{\xi^n}{n(q^{\frac{n}{2}}-q^{-\frac{n}{2}})}\right\}\right) 
        & = (\xi q^{\frac{1}{2}})^{|\lambda|} q^{b(\lambda)} \prod_{(i,j) \in \lambda} \frac{1}{1 - q^{h(i,j)}} \, , \\
        s_\lambda(\{t_n = \xi \delta_{n,1}\}) 
        & = \xi^{|\lambda|} \prod_{(i,j) \in \lambda} \frac{1}{h(i,j)} 
        \, ,
        \end{align}
    \end{subequations}
    from which we conclude the proof.
\end{proof}

\begin{corollary}\label{cor:dDDP}
    The $q$-Plancherel measures are discrete determinantal point processes. 
\end{corollary}

\subsection{$q$-Bessel kernel}

For the squared type measure, we have the following Christoffel--Darboux type formula for the correlation kernel.
\begin{theorem}\label{thm:q-Bessel_kernel}
    Let $z_1, \ldots, z_k \in \mathbb{Z}'$.
    The correlation function of the squared type $q$-Plancherel measure is given by 
    \begin{align}
        \rho(z_1,\ldots,z_k) = \det_{1 \le i, j \le k} K_{q\text{B}}(z_i,z_j) \, ,
    \end{align}
    with the correlation kernel,
    \begin{align}
        K_{q\text{B}}(r,s) = \xi \frac{J_{r+\frac{1}{2}}J_{s-\frac{1}{2}} - J_{r-\frac{1}{2}}J_{s+\frac{1}{2}}}{q^{\frac{1}{2}(r-s)} - q^{-\frac{1}{2}(r-s)}} \, , 
    \end{align}
    where $J_n = J^{(3)}_n(2\xi;q)$ is the Hahn--Exton $q$-Bessel function.
    We call this the $q$-Bessel kernel.
\end{theorem}

\begin{proof}
    Under the specialization~\eqref{eq:sq_specialization}, the auxiliary function $\mathbb{J}(z;t,\tilde{t})$ defined in \eqref{eq:J_aux_fn} is identified with the generating function of the Hahn--Exton $q$-Bessel function $J_n = J^{(3)}_n(2\xi;q)$ (see, e.g.,~\cite{Gasper2004}),
    \begin{align}
    \mathbb{J}\left(z;t,\tilde{t} \right) = \mathbb{J}\left(z;\tilde{t}, t \right)^{-1} = \exp \left( - \sum_{n \ge 1} \frac{\xi^n( z^n - z^{-n} )}{n(q^{\frac{n}{2}}-q^{-\frac{n}{2}})}  \right) = \frac{(\xi q^{\frac{1}{2}}z^{-1};q)_\infty}{(\xi q^{\frac{1}{2}}z;q)_\infty} = \sum_{n \in \mathbb{Z}} (q^{\frac{1}{2}}z)^n J_n \, . \label{eq:SchurJ_generating_function}
    \end{align}
    Hence, by Proposition~\ref{prop:Schur_kernel_series}, the correlation kernel is given by
    \begin{align}
        K(r,s) = q^{\frac{1}{2}(r+s)} \sum_{k \in \mathbb{Z}_{>0}'} q^{k} J_{r+k} J_{s+k} \, , \qquad r, s \in \mathbb{Z}' \, ,
    \end{align}
    which implies
    \begin{align}
        (q^{r} - q^{s}) K(r,s) & = q^{\frac{1}{2}(r+s)} \sum_{k \in \mathbb{Z}_{>0}'} \left( q^{r+k} J_{r+k} J_{s+k} - q^{s+k} J_{r+k} J_{s+k} \right) \, .        
    \end{align}
    Using the recurrence relation of the Hahn--Exton $q$-Bessel function~\eqref{eq:q-Bessel_rec_rel3}, 
    \begin{align}
        J_{n+1} = \left( \frac{1 - q^n}{\xi} + \xi \right) J_n - J_{n-1}\, ,
    \end{align}
    we have
    \begin{align}
        (q^{r} - q^{s}) K(r,s) & = - \xi q^{\frac{1}{2}(r+s)} \sum_{k \in \mathbb{Z}_{>0}'} \left( J_{r+k+1} J_{s+k} + J_{r+k-1} J_{s+k} - J_{r+k} J_{s+k+1} - J_{r+k} J_{s+k-1} \right) 
        \nonumber \\
        & = \xi q^{\frac{1}{2}(r+s)} \left( J_{r+\frac{1}{2}} J_{s-\frac{1}{2}} - J_{r-\frac{1}{2}} J_{s+\frac{1}{2}} \right) \, ,
    \end{align}
    from which we obtain the correlation kernel.
\end{proof}

A crucial fact in the proof is that $\mathbb{J}$ is the generating function of the $q$-Bessel functions~\eqref{eq:SchurJ_generating_function}, which is specific for the squared type measure. 
Such an expression is not obvious for the mixed type measure, hence an analogue of Theorem~\ref{thm:q-Bessel_kernel} is not available for the moment.

\subsubsection*{$q \to 1$ limit}

Set $\xi = (1 - q)\eta$. 
In the limit $q \to 1$, the Hahn--Exton $q$-Bessel function is reduced to the ordinary Bessel function, $\lim_{q \to 1} J_n^{(3)}((1-q)\eta;q) = J_n(\eta)$, which directly implies the reduction of the $q$-Bessel kernel to the discrete Bessel kernel in this limit,
\begin{align}
    \lim_{q \to 1} K_{q\text{B}}(r,s)\Big|_{\xi = (1 - q)\eta} = K_\text{dB}(r,s) = \eta \frac{J_{r-\frac{1}{2}}(2\eta)J_{s+\frac{1}{2}}(2\eta) - J_{r+\frac{1}{2}}(2\eta)J_{s-\frac{1}{2}}(2\eta)}{r-s} \, .
\end{align}

\subsection{Scaling limit}\label{sec:scaling_lim}

As in the case of the discrete Bessel kernel, we consider the scaling limit of the $q$-Bessel kernel, which gives rise to the universal correlation kernels.

\subsubsection*{Bulk scaling limit}
We first consider the bulk scaling limit, where the kernel is asymptotic to the sine kernel.
\begin{proposition}\label{prop:bulk_lim}
    Let $a = - 2 \log(1+\xi)$, $b = - 2\log(1-\xi)$, and $x \in [a,b]$.
    The $q$-Bessel kernel is asymptotic to the sine kernel in the following limit (the bulk scaling limit),
    \begin{align}
        \lim_{q \to 1} K_{q\text{B}}\left( - \frac{x}{\log q} + u, - \frac{x}{\log q} + v \right) = 
        \begin{cases}
            \displaystyle
            \frac{\sin \pi\varrho(u - v)}{\pi(u - v)} & (u \neq v) \\ 
            \varrho & (u = v)
        \end{cases}
    \end{align}
    where $\varrho = \varrho(x) \in [0,1]$ is the bulk one-point function,
    \begin{align}
    \varrho(x) = \frac{1}{\pi} \arccos \frac{1}{2} \left( \xi + \frac{1 - \ee^{-x}}{\xi} \right) \, , \quad x \in [a,b] \, . \label{eq:bulk_1pt_fn}
    \end{align}    
\end{proposition}
We may apply the same approach as, e.g.,~\cite[\S3]{Okounkov:2002}, to prove this statement. 
We here provide a sketch of the proof:
Set $q = \ee^{-\epsilon}$ with $\epsilon > 0$.
From the contour integral formula of the correlation kernel \eqref{eq:Schur_corr_kernel_int}, we have
\begin{align}
    K_{q\text{B}}\left( r, s \right) = \frac{1}{(2 \pi \ii)^2} \oint_{|z| > |w|} \frac{\ee^{S(z,r) - S(w,s)}}{z - w} \frac{\dd{z}\dd{w}}{z^\frac{1}{2} w^{\frac{1}{2}}} \, ,
\end{align}
where we define the action function,
\begin{align}
    S(z,r) = - \sum_{n \ge 1} \frac{\xi^n(z^n - z^{-n})}{n(q^{\frac{n}{2}} - q^{-\frac{n}{2}})} - r \log z = \frac{1}{\epsilon} \sum_{n \ge 1} \frac{\xi^n(z^n - z^{-n})}{n^2} - r \log z + O(\epsilon) \, .
\end{align}
Therefore, we may apply the method of steepest descent in the limit $\epsilon \to 0$ ($q \to 1$).
Parametrizing $z = \ee^{\ii \theta}$, we have
\begin{align}
    \ii S(\ee^{\ii \theta},r) = 2 \sum_{n \ge 1} \frac{\xi^n \sin n \theta}{n(q^{\frac{n}{2}} - q^{-\frac{n}{2}})} + r \theta = - \frac{2}{\epsilon} \sum_{n \ge 1} \frac{\xi^n \sin n \theta}{n^2} + r \theta + O(\epsilon) \, .
\end{align}
Set
\begin{align}
    \alpha = - 2 \sum_{n \ge 1} \frac{\xi^n}{q^{\frac{n}{2}} - q^{-\frac{n}{2}}} = \sum_{k=0}^\infty \alpha_k \epsilon^{k-1} \, , \qquad
    \beta = - \sum_{n \ge 1}  \frac{n^2 \xi^n}{q^{\frac{n}{2}} - q^{-\frac{n}{2}}} = \sum_{k = 0}^\infty \beta_k \epsilon^{k-1} \, .
\end{align}
Note that $\alpha, \beta \ge 0$.
In particular, the leading contributions are given by
\begin{align}
    \alpha_0 = 2 \sum_{n \ge 1} \frac{\xi^n}{n} = - 2 \log(1 - \xi) \ge 0 \, , \qquad 
    \beta_0 = \sum_{n \ge 1} n \xi^n = \frac{\xi}{(1 - \xi)^2} \ge 0 \, .
\end{align}
We compute
\begin{align}
    \ii \pdv{}{\theta} S(\ee^{\ii \theta},r) = 2 \sum_{n \ge 1} \frac{\xi^n \cos n \theta}{q^{\frac{n}{2}} - q^{-\frac{n}{2}}} + r = - \frac{2}{\epsilon} \sum_{n \ge 1} \frac{\xi^n}{n} \cos n \theta + r + O(\epsilon) \, , \label{eq:d_S}
\end{align}
and hence the stationary equation in the limit $\epsilon \to 0$ ($q \to 1$) is given by
\begin{align}
    \epsilon r & \stackrel{O(\epsilon^2)}{=} 2 \sum_{n \ge 1} \frac{\xi^n}{n} \cos n \theta = \sum_{n \ge 1} \frac{\xi^n}{n} (\ee^{\ii n \theta} + \ee^{-\ii n \theta}) \nonumber \\
    & = - \log (1 - \xi \ee^{\ii \theta}) - \log (1 - \xi \ee^{- \ii \theta}) = - \log (1 + \xi^2 - 2 \xi \cos \theta) \, .
\end{align}
Therefore, $\varrho$ solves the following equation,
\begin{align}
    x = 2 \sum_{n \ge 1} \frac{\xi^n}{n} \cos n \pi \varrho = - \log (1 + \xi^2 - 2 \xi \cos \pi \varrho) \, ,
\end{align}
from which we obtain
\begin{align}
    \cos \pi \varrho = \frac{1}{2} \left( \xi + \frac{1 - \ee^{-x}}{\xi} \right) \, .
\end{align}
We remark that such a solution exists only when $\epsilon r \in [a,b]$.
Under this condition, we may apply the method of steepest descent to obtain the sine kernel.
See, e.g.,~\cite[\S3]{Okounkov:2002}, for details.

\subsubsection*{Limit shape}

We have found the bulk one-point function~\eqref{eq:bulk_1pt_fn} for $x \in [a,b]$. 
After simple analysis, we can similarly consider $x \not\in [a,b]$ to have
\begin{align}
    \varrho(x) =
    \begin{cases}
        1 & (x < a) \\
        \displaystyle
        \frac{1}{\pi} \arccos \frac{1}{2} \left( \xi + \frac{1 - \ee^{-x}}{\xi} \right) & (a \le x \le b) \\
        0 & (b < x)
    \end{cases}
\end{align}
The domain $x < a$ is called the frozen region, while $x > b$ is called the empty region.
We then obtain the profile function $\Omega$ from the one-point function.
We remark that $\lim_{\xi \to 1} a = - 2 \log 2 = -1.386...$ and $\lim_{\xi \to 1} b = \infty$.
Recalling the relation $\Omega' = 1 - 2 \varrho$, we have
\begin{align}
    \Omega(x) = 
    \begin{cases}
        \displaystyle x - 2a - 2 \int_{a}^x \varrho(z) \dd{z} & (a \le x \le b) \\
        |x| & (x < a, b < x)
    \end{cases}
\end{align}
We plot the one-point function for the squared type measure in Fig.~\ref{fig:rho_plot} and the profile function in Fig.~\ref{fig:omega_plot} by changing the parameter for $\xi \in [0,1)$.
The profile becomes more asymmetric for larger $\xi$.

\begin{figure}[t]
    \centering
    \begin{subfigure}[b]{0.45\textwidth}
        \centering
        \begin{tikzpicture}
        \node at (0,0) {\includegraphics[width=\linewidth]{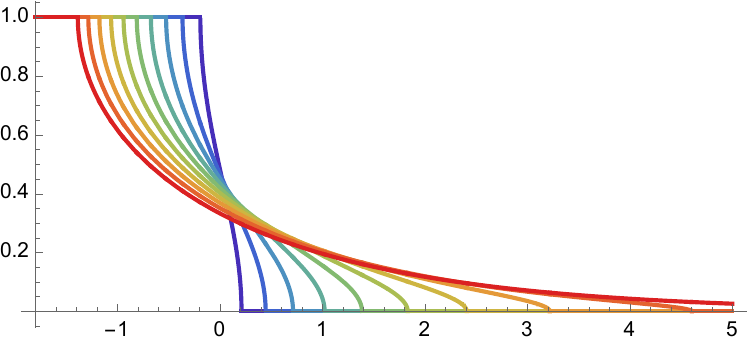}};
        \end{tikzpicture}
        \caption{One-point function $\varrho$.}
        \label{fig:rho_plot}
     \end{subfigure}
     \hfill
     \begin{subfigure}[b]{0.5\textwidth}
        \centering
        \begin{tikzpicture}
        \node at (0,0) {\includegraphics[width=.8\linewidth]{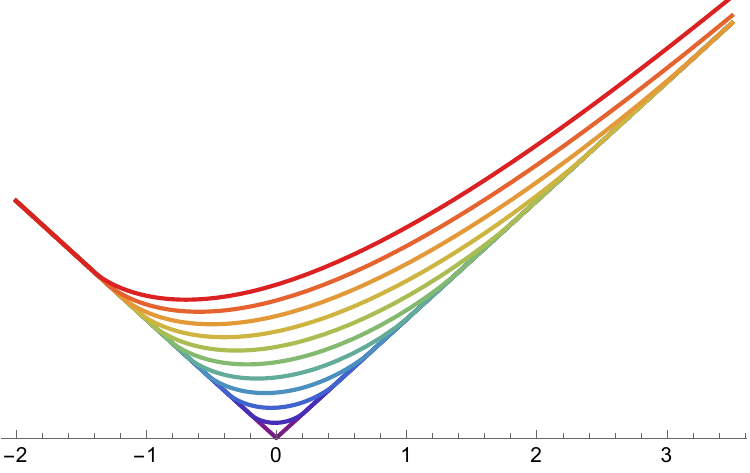}};
        \node at (4.2,.5) {\includegraphics[height=9em]{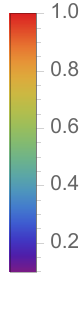}};
        \node at (4.,2.5) {$\xi$};
        \end{tikzpicture}
        \caption{Profile function $\Omega$.}
        \label{fig:omega_plot}
     \end{subfigure}
     \caption{The limit shape functions for the parameter $\xi \in [0,1)$ (step 0.1) for the squared type measure.}
     \label{fig:prof_fn}
\end{figure}

\subsubsection*{Edge scaling limit}
The Airy kernel is the universal correlation kernel appearing in the so-called edge scaling limit of GUE,
\begin{align}
    K_\text{Ai}(x,y) = \int_0^\infty \operatorname{Ai}(x+z) \operatorname{Ai}(y+z) \dd{z} = \frac{\operatorname{Ai}(x) \operatorname{Ai}'(y) - \operatorname{Ai}'(x) \operatorname{Ai}(y)}{x-y} \, ,
\end{align}
where the Airy function is defined by
\begin{align}
    \operatorname{Ai}(x) = \frac{1}{2 \pi \ii} \int_{\ii \mathbb{R}} \ee^{\frac{z^3}{3} - x z } \dd{z} \, .
\end{align}

\begin{proposition}\label{prop:edge_lim}
    The $q$-Bessel kernel is asymptotic to the Airy kernel in the following limit (the edge scaling limit),
    \begin{align}
        \lim_{q \to 1} \left(-\frac{\beta_0}{\log q}\right)^{\frac{1}{3}} K_{q\text{B}}\left( - \frac{\alpha_0}{\log q} + \left(-\frac{\beta_0}{\log q}\right)^{\frac{1}{3}} x, - \frac{\alpha_0}{\log q} + \left(-\frac{\beta_0}{\log q}\right)^{\frac{1}{3}} y \right) = K_\text{Ai}(x,y) \, ,
    \end{align}
    with $\alpha_0 = - 2 \log (1 - \xi)$ and $\beta_0 = \xi/(1-\xi)^2$.
\end{proposition}
We may apply the method of steepest descent as before.
Set $q = \ee^{-\epsilon}$.
For $r = \alpha_0/\epsilon$, $\theta=0$ is the critical point of the action. 
Expanding the action around the critical point, we obtain
\begin{align}
    \ii S\left(\ee^{\ii \theta},\frac{\alpha_0}{\epsilon} + \left(\frac{\beta_0}{\epsilon}\right)^{\frac{1}{3}} x \right) \stackrel{O(\epsilon)}{=} \frac{\beta_0}{\epsilon} \frac{\theta^3}{3} + \left(\frac{\beta_0}{\epsilon}\right)^{\frac{1}{3}} x \theta + O(\theta^5) \, .
\end{align}
Applying the change of variable, $\theta = (\beta_0/\epsilon)^{-\frac{1}{3}} \tilde{\theta}$, together with the change of the contour, we obtain the Airy kernel in the limit $\epsilon \to 0$ ($q \to 1$).
See, e.g.,~\cite[\S3]{Okounkov:2002}, for details.

\section{Gap probability and orthogonal polynomials}

\subsection{Toeplitz determinant formula}

Let $I \subset \mathbb{Z}'$.
The probability to find no element in $I$ is called the gap probability, which is one of the most important quantities associated with the measure on partitions.
In particular, for the determinantal case, the gap probability for the interval $I = [N+\frac{1}{2},\infty)$ is given by the following Fredholm determinant of the correlation kernel,
\begin{align}
    \mathbb{P}[\lambda_1 \le N] = \det(1 - K)_{\ell^2([N+\frac{1}{2},\infty))} \, .
\end{align}
For the Schur measure, Borodin and Okounkov~\cite{Borodin2000} obtained the Toeplitz determinant formula for the gap probabilities.
In this Section, we derive such a Toeplitz determinant formula for the $q$-Plancherel measure.
In the following, we focus on the squared type measure.

\begin{proposition}\label{prop:gap_prob_unitary_matrix}
    We have the following unitary matrix integral formula for the gap probabilities for the squared type $q$-Plancherel measure,
    \begin{align}
        \mathbb{P}_{q\text{PP}}[\ell(\lambda) \le N] = \frac{Z_N}{Z_\infty}
        \, , \qquad 
        \mathbb{P}_{q\text{PP}}[\lambda_1 \le N] = \frac{\check{Z}_N}{\check{Z}_\infty} \, ,
        \label{eq:gap_prob_unitary_matrix}
    \end{align}
    where we define the integral over the unitary group,
    \begin{align}
        Z_N = \int_{\mathrm{U}(N)} \det\mathbb{I}(U;\xi) \dd{U} \, , \qquad 
        \check{Z}_N = \int_{\mathrm{U}(N)} \det\check{\mathbb{I}}(U;\xi) \dd{U} \, ,
    \end{align}
    with
    \begin{subequations}\label{eq:I_definition}
    \begin{align}
        \mathbb{I}(z;\xi) & = (\xi q^{\frac{1}{2}} z, \xi q^{\frac{1}{2}} z^{-1};q)_\infty^{-1} = \exp\left( - \sum_{n \ge 1} \frac{\xi^{n}(z^n + z^{-n})}{n(q^{\frac{n}{2}} - q^{-\frac{n}{2}})} \right) \, , \\
        \check{\mathbb{I}}(z;\xi) & = (-\xi q^{\frac{1}{2}} z, -\xi q^{\frac{1}{2}} z^{-1}; q)_{\infty} = \exp\left( \sum_{n \ge 1} \frac{(-\xi)^{n}(z^n + z^{-n})}{n(q^{\frac{n}{2}} - q^{-\frac{n}{2}})} \right) \, ,
    \end{align}        
    \end{subequations}
    and $Z_\infty = \check{Z}_\infty = M(\xi,q)$.
\end{proposition}

\begin{proof}
    Diagonalizing the unitary matrix, we have
    \begin{align}
        Z_N & = \frac{1}{N!} \oint_{\mathbb{T}_N} \prod_{1 \le i < j \le N} |z_i - z_j|^2 \prod_{i=1}^N (\xi q^{\frac{1}{2}} z_i, \xi q^{\frac{1}{2}} z_i^{-1}; q)_{\infty}^{-1} \prod_{i=1}^N \frac{\dd{z}_i}{2 \pi \ii z_i} \nonumber \\
        & = \frac{1}{N!} \oint_{\mathbb{T}_N} \prod_{1 \le i < j \le N} |z_i - z_j|^2 \prod_{i=1}^N \dd{\mu}(z_i) 
    \end{align}
    where we denote by $\mathbb{T}_N$ the $N$-dimensional torus, and we define the measure on the unit circle, $\dd{\mu}(z) = \mathbb{I}(z;\xi) \frac{\dd{z}}{2 \pi \ii z}$.
    Similarly, defining $\mathrm{d}\check{\mu}(z) = \check{\mathbb{I}}(z;\xi) \frac{\dd{z}}{2 \pi \ii z}$, we have
    \begin{align}
        \check{Z}_N & = \frac{1}{N!} \oint_{\mathbb{T}_N} \prod_{1 \le i < j \le N} |z_i - z_j|^2 \prod_{i=1}^N \dd{\check{\mu}}(z_i) \, .
    \end{align}
    Applying the Cauchy identities~\eqref{eq:Cauchy_ids}, we have
    \begin{align}
        & \prod_{i=1}^N \mathbb{I}(z_i;\xi) = \exp\left( - \sum_{i=1}^N \sum_{n \ge 1} \frac{\xi^{n}(z_i^n + z_i^{-n})}{n(q^{\frac{n}{2}} - q^{-\frac{n}{2}})} \right) \nonumber \\ 
        & = \sum_{\lambda} s_\lambda(z_1,\ldots,z_N) s_\lambda\left( \left\{t_n = - \frac{\xi^n}{n(q^{\frac{n}{2}}-q^{-\frac{n}{2}})}\right\}\right) \sum_{\lambda'} s_{\lambda'}(z_1^{-1},\ldots,z_N^{-1}) s_{\lambda'}\left( \left\{t_n = - \frac{\xi^n}{n(q^{\frac{n}{2}}-q^{-\frac{n}{2}})}\right\}\right) \, .
    \end{align}
    Note that $s_\lambda(z_1,\ldots,z_N) = 0$ for $\ell(\lambda) > N$.
    Hence the partition sum is restricted to $\lambda$, $\lambda' \in \mathsf{P}_N$.
    Similarly, we have
    \begin{align}
        & \prod_{i=1}^N \check{\mathbb{I}}(z_i;\xi) = \exp\left( \sum_{i=1}^N \sum_{n \ge 1} \frac{(-\xi)^{n}(z_i^n + z_i^{-n})}{n(q^{\frac{n}{2}} - q^{-\frac{n}{2}})} \right) \nonumber \\ 
        & = \sum_{\lambda} s_{\lambda^{\rmT}}(z_1,\ldots,z_N) s_\lambda\left( \left\{t_n = - \frac{\xi^n}{n(q^{\frac{n}{2}}-q^{-\frac{n}{2}})}\right\}\right) \sum_{\lambda'} s_{\lambda'^{\rmT}}(z_1^{-1},\ldots,z_N^{-1}) s_{\lambda'}\left( \left\{t_n = - \frac{\xi^n}{n(q^{\frac{n}{2}}-q^{-\frac{n}{2}})}\right\}\right) \, .
    \end{align}    
    In this case, the sum is restricted to $\lambda^{\rmT}$, $\lambda'^{\rmT} \in \mathsf{P}_N$.    
    Then, by the orthonormal condition of the Schur polynomials~\eqref{eq:Schur_orthonormal}, we obtain
    \begin{subequations}
        \begin{align}
            Z_N & = \sum_{\lambda \in \mathsf{P}_N} s_\lambda\left( \left\{t_n = - \frac{\xi^n}{n(q^{\frac{n}{2}}-q^{-\frac{n}{2}})}\right\}\right) s_{\lambda}\left( \left\{t_n = - \frac{\xi^n}{n(q^{\frac{n}{2}}-q^{-\frac{n}{2}})}\right\}\right) = Z_\infty \mathbb{P}_{q\text{PP}}[\ell(\lambda) \le N] \, , \\
            \check{Z}_N & = \sum_{\lambda^{\rmT} \in \mathsf{P}_N} s_\lambda\left( \left\{t_n = - \frac{\xi^n}{n(q^{\frac{n}{2}}-q^{-\frac{n}{2}})}\right\}\right) s_{\lambda}\left( \left\{t_n = - \frac{\xi^n}{n(q^{\frac{n}{2}}-q^{-\frac{n}{2}})}\right\}\right) = \check{Z}_\infty \mathbb{P}_{q\text{PP}}[\ell(\lambda^{\rmT}) \le N] \, .
        \end{align}
    \end{subequations}
    Recalling $\ell(\lambda^{\rmT}) = \lambda_1$ for the latter case, we obtain the formula~\eqref{eq:gap_prob_unitary_matrix}.
\end{proof}

The weight functions $\mathbb{I}(z;\xi)$ and $\check{\mathbb{I}}(z;\xi)$ take positive real values on the unit circle $z \in \mathbb{T}$ for $q, \xi \in [0,1)$.
Putting $z = \ee^{\ii \theta}$ in the definition~\eqref{eq:I_definition}, we have
\begin{align}
    \mathbb{I}(\ee^{\ii \theta};\xi) = \prod_{n = 0}^\infty (1 + \xi^2 q^{2n+1} - 2 \xi q^{n+\frac{1}{2}} \cos \theta)^{-1} \, , \quad
    \check{\mathbb{I}}(\ee^{\ii \theta};\xi) = \prod_{n = 0}^\infty (1 + \xi^2 q^{2n+1} + 2 \xi q^{n+\frac{1}{2}} \cos \theta) \, .
\end{align}
Under this parametrization, we have an alternative expression of the unitary matrix integrals,
\begin{subequations}
    \begin{align}
        Z_N & = \frac{1}{N!} \int_{-\pi}^\pi \prod_{1 \le i < j \le N} \left( 2 \sin \left( \frac{\theta_i - \theta_j}{2} \right) \right)^2 \prod_{i=1}^N \mathbb{I}(\ee^{\ii \theta_i};\xi) \frac{\dd{\theta}_i}{2 \pi} \, , \\
        \check{Z}_N & = \frac{1}{N!} \int_{-\pi}^\pi \prod_{1 \le i < j \le N} \left( 2 \sin \left( \frac{\theta_i - \theta_j}{2} \right) \right)^2 \prod_{i=1}^N \check{\mathbb{I}}(\ee^{\ii \theta_i};\xi) \frac{\dd{\theta}_i}{2 \pi} \, ,        
    \end{align}
\end{subequations}

\begin{proposition}\label{prop:Toeplitz_Z}
    The unitary matrix integrals are given by the Toeplitz determinants,
    \begin{align}
        Z_N = \det_{0 \le i, j \le N-1} I_{-i+j} \, , \qquad 
        \check{Z}_N = \det_{0 \le i, j \le N-1} \check{I}_{-i+j} \, .
    \end{align}
    where $I_n = I^{(1)}_n(2\xi q^{\frac{1}{2}};q)$ and $\check{I}_n = q^{\frac{n^2}{2}} I^{(2)}_n(2\xi;q)$ are the modified $q$-Bessel functions (see \S\ref{sec:q-Bessel} for details).
\end{proposition}

\begin{proof}
    From the generating functions of the modified $q$-Bessel function \eqref{eq:Bessel_gen_fn_mod}, we have
    \begin{align}
        \mathbb{I}(z;\xi) = \sum_{n \in \mathbb{Z}} z^n I_n \, , \qquad 
        \check{\mathbb{I}}(z;\xi) = \sum_{n \in \mathbb{Z}} z^n \check{I}_n \, .
    \end{align}
    Recalling $\prod_{1 \le i < j \le N}|z_i - z_j|^2 = \det_{1 \le i, j \le N} (z_i^{j-1}) \det_{1 \le i, j \le N} (z_i^{-j+1})$ for $z_1,\ldots,z_N \in \mathbb{T}$, we may use Andréief's formula to evaluate the integral as follows,
    \begin{subequations}
    \begin{align}
        Z_N & = \det_{0 \le i, j \le N-1} \left( \oint_{\mathbb{T}} z^{i-j} \dd{\mu}(z) \right) = \det_{0 \le i, j \le N - 1} I_{-i+j} \, , \\
        \check{Z}_N & = \det_{0 \le i, j \le N-1} \left( \oint_{\mathbb{T}} z^{i-j} \dd{\check{\mu}}(z) \right) = \det_{0 \le i, j \le N - 1} \check{I}_{-i+j} \, .
    \end{align}        
    \end{subequations}
\end{proof}

\begin{remark}\label{rmk:inversion_symmetry_I}
    Since $\mathbb{I}(z;\xi) = \mathbb{I}(z^{-1};\xi)$ and $\check{\mathbb{I}}(z;\xi) = \check{\mathbb{I}}(z^{-1};\xi)$, we have the symmetry, $I_n = I_{-n}$ and $\check{I}_n = \check{I}_{-n}$.
\end{remark}

In the $q$-deformed setup, we have two different Toeplitz determinant formulas for the gap probability for $\lambda_1 < N$ and $\lambda_1^{\rmT} < N$ (see Fig.~\ref{fig:prof_fn}). 
In the case of the ordinary Poissonized Plancherel measure, these two situations are equivalent due to the symmetry of the measure under $\eta \leftrightarrow -\eta$ in \eqref{eq:PP_measure}.
In fact, both $I_n = I_n^{(1)}(2\xi q^{\frac{1}{2}};q)$ and $\check{I}_n = q^{\frac{n^2}{2}} I_n^{(2)}(2\xi;q)$ are reduced to the ordinary modified Bessel function $I_n(2\eta)$ in the limit $q \to 1$ under the parametrization $\xi = (1-q)\eta$.

\subsubsection*{Hypergeometric formula}
The hypergeometric form of the modified $q$-Bessel function can be obtained using the residue formula.
Recall that the weight function $\mathbb{I}$ has poles at $\xi q^{\mathbb{Z}'_{> 0}}$ in the domain $|z| < 1$ and $\xi^{-1} q^{\mathbb{Z}'_{< 0}}$ in the domain $|z| > 1$.
Hence, the integral on the unit circle is given by the residues of the former poles,
\begin{align}
    I_n = I_{-n} = \oint_{\mathbb{T}} z^{n} d\mu(z) 
    & = \frac{\xi^n q^{\frac{n}{2}}}{(q,\xi^2 q;q)_\infty} \sum_{m=0}^\infty \frac{(\xi^2 q;q)_m}{(q;q)_m} (-1)^m q^{m \choose 2} q^{(n+1)m} \nonumber \\
    & = \frac{\xi^n q^{\frac{n}{2}}}{(q,\xi^2 q;q)_\infty} {}_1\phi_1\left(\begin{matrix}
        \xi^2 q \\ 0
    \end{matrix};q;q^{n+1}\right) \, ,
\end{align}
which agrees with the hypergeometric formula of $I_n^{(1)}(2 \xi q^{\frac{1}{2}};q)$~\eqref{eq:hypergeometric_I}.

\subsection{Orthogonal polynomials on the unit circle}\label{sec:OPUC}

We define inner products on the unit circle, 
\begin{align}
    \left< f, g \right>_{\mathbb{I}} = \oint_{\mathbb{T}} f(z) \overline{g(z)} \dd{\mu}(z) \, , \qquad 
    \left< f, g \right>_{\check{\mathbb{I}}} = \oint_{\mathbb{T}} f(z) \overline{g(z)} \dd{\check{\mu}}(z) \, .
\end{align}
Under these inner products, we define orthonormal polynomials,
\begin{align}
    \left< p_n, p_m \right>_{\mathbb{I}} = \delta_{n,m} \, , \qquad 
    \left< \check{p}_n, \check{p}_m \right>_{\check{\mathbb{I}}} = \delta_{n,m} \, ,  
\end{align}
where
\begin{align}
    p_n(z) = \kappa_n z^n + \cdots \, , \quad 
    \check{p}_n(z) = \check{\kappa}_n z^n + \cdots \, , \quad \kappa_n,\, \check{\kappa}_n > 0 \, .
\end{align}
We denote the corresponding monic polynomials by
\begin{align}
    \pi_n(z) = \frac{1}{\kappa_n} p_n(z) \, , \qquad
    \check{\pi}_n(z) = \frac{1}{\check{\kappa}_n} \check{p}_n(z) \, ,
\end{align}
where $\pi_0 = \check{\pi}_0 =1$.
Since the weight functions are symmetric, $\mathbb{I}(z;\xi) = \mathbb{I}(z^{-1};\xi)$ and $\check{\mathbb{I}}(z;\xi) = \check{\mathbb{I}}(z^{-1};\xi)$, all the coefficients of these orthogonal polynomials are real, i.e., $\overline{p_n(z)} = p_n(z^{-1})$ and $\overline{\check{p}_n(z)} = \check{p}_n(z^{-1})$ for $z \in \mathbb{T}$.

\begin{lemma}
The unitary matrix integrals are given by the coefficients of the orthogonal polynomials,
\begin{align}
    Z_N = \prod_{n=0}^{N-1} \kappa_n^{-2} \, , \qquad
    \check{Z}_N = \prod_{n=0}^{N-1} \check{\kappa}_n^{-2} \, ,
\end{align}
and thus
\begin{align}
\kappa_n^2 = \frac{Z_{n}}{Z_{n+1}} \, , \qquad 
\check{\kappa}_n^2 = \frac{\check{Z}_{n}}{\check{Z}_{n+1}} \, , \qquad n \ge 0 \, .
\end{align}
\end{lemma}
\begin{proof}
    See, e.g., \cite{Mehta:2004RMT,Forrester:2010,Eynard:2015aea}.
\end{proof}

The monic orthogonal polynomials are given by the unitary matrix integral with the characteristic polynomial insertion (Heine's formula),
\begin{subequations}
\begin{align}
    \pi_n(z) & = \frac{1}{Z_n} \int_{\mathrm{U}(n)} \det(z - U) \det \mathbb{I}(U;\xi) \dd{U} \, , \\
    {\check\pi}_n(z) & = \frac{1}{\check{Z}_n} \int_{\mathrm{U}(n)} \det(z - U) \det \check{\mathbb{I}}(U;\xi) \dd{U} \, .
\end{align}    
\end{subequations}
Let $x_n = \pi_n(0)$, $y_n = \check{\pi}_n(0)$.
Then, we have the following determinantal formula,
\begin{subequations}
\begin{align}
x_n & = \frac{(-1)^n}{Z_n} \int_{\mathrm{U}(n)} \det U \det \mathbb{I}(U;\xi) \dd{U} = \frac{(-1)^n}{Z_n} \det_{0 \le i, j \le n - 1} I_{-i+j-1} \, , \\ 
y_n & = \frac{(-1)^n}{\check{Z}_n} \int_{\mathrm{U}(n)} \det U \det \check{\mathbb{I}}(U;\xi) \dd{U} = \frac{(-1)^n}{\check{Z}_n} \det_{0 \le i, j \le n - 1} \check{I}_{-i+j-1} \, .
\end{align}    
\end{subequations}
As mentioned in Remark~\ref{rmk:inversion_symmetry_I}, we may also realize $x_n$ and $y_n$ from the unitary matrix integral with insertion of $\det U^{-1}$ instead of $\det U$ due to the symmetry of the weight function, $\mathbb{I}(z;\xi) = \mathbb{I}(z^{-1};\xi)$.

\begin{lemma}\label{lemma:Zx_relation}
    The following relations hold,
\begin{align}
\frac{\kappa_{n-1}^2}{\kappa_n^2} = \frac{Z_{n+1}Z_{n-1}}{Z_n^2} = 1 - x_n^2 \, , \qquad 
\frac{{\check\kappa}_{n-1}^2}{\check{\kappa}_n^2} = \frac{\check{Z}_{n+1}\check{Z}_{n-1}}{\check{Z}_n^2} = 1 - y_n^2 \, .
\end{align}
\end{lemma}
\begin{proof}
    See, e.g.,~\cite[\S1.1]{Adler:2003CMP}.
\end{proof}

We will derive the non-linear recurrence relations for $x_n$ and $y_n$ from the Riemann--Hilbert problem associated with the orthogonal polynomial on the unit circle in \S\ref{sec:RHP}.

\subsection{Relation to $q$-orthogonal polynomials}\label{sec:q-orthogonal_polynomial}

It has been known that all the classical $q$-orthogonal polynomials are obtained from the Askey--Wilson polynomial $\{p_n = p_n(\cdot;a,b,c,d;q)\}_{n\in\mathbb{Z}_{\ge 0}}$, which is the classical $q$-orthogonal polynomial obeying the following condition (e,g., \cite{Ismail2005,Koekoek2010}),  
\begin{align}
    \frac{1}{2\pi} \int_{-1}^1 p_n(x) p_m(x) w(x) \frac{\dd{x}}{\sqrt{1-x^2}} = h_n \delta_{n,m} \, ,
\end{align}
where the weight function is given by
\begin{align}
    w(x) = w(x;a,b,c,d;q) = \frac{h(x,1)h(x,-1)h(x,q^{\frac{1}{2}})h(x,-q^{\frac{1}{2}})}{h(x,a)h(x,b)h(x,c)h(x,d)} \, , \quad |a|, |b|, |c|, |d| < 1 \, ,
\end{align}
with
\begin{align}
    h(x,\alpha) = \prod_{n=0}^\infty (1 + \alpha^2 q^{2n} - 2 \alpha q^{n} x) \, ,
\end{align}
and 
\begin{align}
    h_n = \frac{(abcd q^{n-1};q)_n (abcd q^{2n};q)_\infty}{(q^{n+1},abq^n,acq^n,adq^n,bcq^n,bdq^n,cdq^n;q)_\infty} \, .
\end{align}
Our orthogonal polynomial have a similar structure to this $q$-orthogonal polynomial as follows.

Due to the symmetry of the weight function $\mathbb{I}(z;\xi) = \mathbb{I}(z^{-1};\xi)$, we may instead consider a modified inner product with respect to the weight function,
\begin{align}
    \left< f, g \right> = \oint_{\mathbb{T}} f\left(\frac{z+z^{-1}}{2}\right) \overline{g\left(\frac{z+z^{-1}}{2}\right)} \mathbb{I}(z;\xi) \frac{\dd{z}}{2 \pi \ii z} \, .
\end{align}
Applying the change of variable $z = \ee^{\ii \theta}$, we have
\begin{align}
    \left< f, g \right> & = \int_{-\pi}^\pi f\left(\cos \theta\right) \overline{g}\left(\cos\theta\right) 
    \left( \prod_{n = 0}^\infty (1 + \xi^2 q^{2n+1} - 2 \xi q^{n+\frac{1}{2}} \cos \theta) \right)^{-1} \frac{\dd{\theta}}{2\pi} \nonumber \\
    & = \frac{1}{\pi} \int_{-1}^1 f\left(x\right) \overline{g}\left(x\right) h(x,\xi q^{\frac{1}{2}})^{-1} \frac{\dd{x}}{\sqrt{1 - x^2}} \, , \qquad x = \cos \theta \, .
\end{align}
Compared with the weight function of the Askey--Wilson polynomial, our weight function $h(x;\xi q^{\frac{1}{2}})^{-1}$ cannot be obtained from a naive reduction of that of the Askey--Wilson polynomial, implying that our orthogonal polynomials are not of the classical type.
In fact, our weight function is interpreted as a $q$-analogue of the Bessel transform of the orthogonal polynomial measure on the unit circle~\cite{Golinskii2006}.

Applying the residue formula, we may also write the inner product as the basic hypergeometric series,
\begin{align}
  \left< f, g \right>
  & = \frac{1}{(\xi^2 q;q)_\infty (q;q)_\infty} \sum_{n = 0}^\infty f\left(z_n\right) \bar{g}\left(z_n\right) \frac{(\xi^2 q;q)_n}{(q;q)_n} (-1)^n q^{n \choose 2} q^n \, , \label{eq:Jackson_form}
\end{align}
where we define 
\begin{align}
    z_n = \frac{1}{2} \left( \xi q^{n+\frac{1}{2}} + \xi^{-1} q^{-n-\frac{1}{2}} \right) \, .
\end{align}
This expression is also analogous to the orthogonality condition for the discrete $q$-orthogonal polynomial, also called the $q$-hypergeometric type orthogonal polynomial, e.g., $q$-Racah polynomial, $q$-Hahn polynomial.
This also implies a $q$-analogue of the Bessel transform for the discrete $q$-orthogonal polynomial.

\section{Riemann--Hilbert problem}\label{sec:RHP}

In this Section, we study the Riemann--Hilbert problem associated with the orthogonal polynomials on the unit circle defined in \S\ref{sec:OPUC}.

\subsubsection*{Dual polynomial}
Let $\mathsf{p}$ be a polynomial in $z \in \mathbb{T}$ of degree $n$.
We define a dual polynomial of $\mathsf{p}$ by
\begin{align}
  \mathsf{p}^*(z) = \overline{\mathsf{p}}(z^{-1}) z^n \, .
\end{align}

\subsubsection*{Cauchy transform}
We define the Cauchy transform of $\mathsf{p}$,
\begin{align}
    \mathsf{C}[\mathsf{p}](z) = \frac{1}{2 \pi \ii} \oint_{\mathbb{T}} \frac{\mathsf{p}(w)}{w-z} \dd{w} = \frac{1}{2 \pi \ii} \oint_{\mathbb{T}} \frac{\mathsf{p}(w)}{1-z/w} \frac{\dd{w}}{w} \, .
\end{align}
Since this Cauchy transform is defined by the integral on $\mathbb{T}$, it may cause a singularity, which plays an important role in the Riemann--Hilbert problem as discussed below.

\subsection{Weight function $\mathbb{I}$}

We focus on the weight function $\mathbb{I}$ and discuss the other case $\check{\mathbb{I}}$ in \S\ref{sec:Icheck}.
Here is a useful result about the Riemann--Hilbert problem associated with the orthogonal polynomial on the unit circle.
\begin{proposition}[\cite{Baik:1999JAMS}]
    Let $\mathbb{I}(z) = \mathbb{I}(z,\xi)$.
    Under the notation above, the matrix-valued function,
    \begin{align}
        Y(z) = Y_n(z) = 
        \begin{pmatrix}
            \pi_n(z) & \mathsf{C}[w^{-n}\pi_n(w)\mathbb{I}(w)](z) \\
            - \kappa_{n-1}^2 \pi_{n-1}^*(z) & - \kappa_{n-1}^2 \mathsf{C}[w^{-n}\pi_{n-1}^*(w) \mathbb{I}(w)](z)
        \end{pmatrix} \, ,
    \end{align}
    solves the following Riemann--Hilbert problem: 
    \begin{enumerate}
        \item $Y$ is analytic and $\det Y = 1$ on $\mathbb{C}\backslash\mathbb{T}$.
        
        \item The boundary value on $\mathbb{T}$ from the inside $Y_+$ and from the outside $Y_-$ have the relation,
        \begin{align}
            Y_+(z) = Y_-(z) \begin{pmatrix} 1 & z^{-n} \mathbb{I}(z) \\ 0 & 1 \end{pmatrix} \, .
        \end{align}
        
        \item The asymptotic behavior is given by $\displaystyle Y(z) 
        \begin{pmatrix}
            z^{-n} & 0 \\ 0 & z^n
        \end{pmatrix} = I + O(z^{-1})$ as $z \to \infty$. 
    \end{enumerate}
\end{proposition}
We write the asymptotic expansion of $Y$ as follows,
\begin{align}
    Y_n(z) & = 
    \begin{cases}
        \displaystyle 
        \left( I + \sum_{k \ge 1} z^{-k} Y_{n;k}^{(\infty)} \right) \begin{pmatrix}
        z^{n} & 0 \\ 0 & z^{-n}
        \end{pmatrix} & (z \to \infty)  \\[.5em]
        \displaystyle
        Y_n(0) + \sum_{k \ge 1} z_k Y_{n;k}^{(0)}  & (z \to 0) 
    \end{cases}
\end{align}
We prepare the following Lemma.
\begin{lemma}[\cite{Baik2016}]
We have the inversion relation,
\begin{align}
    Y_n(z) = \sigma_3 Y_n(0)^{-1} Y_n(z^{-1}) z^{n \sigma_3} \sigma_3 \, ,
    \label{eq:Y_inversion}
\end{align}
with $\sigma_3 = \operatorname{diag}(+1,-1)$.
In particular, we have
\begin{align}
    Y_n(0) = \sigma_3 Y_n(0)^{-1} \sigma_3 \, ,
    \label{eq:Y0_inversion}
\end{align}
and also
\begin{align}
    Y_n(0) = 
    \begin{pmatrix}
        x_n & \kappa_n^{-2} \\ - \kappa_{n-1}^2 & x_n
    \end{pmatrix}
    \, .
\end{align}
\end{lemma}

\begin{proposition}
    The matrix-valued function,
    \begin{align}
        \Psi(z) = \Psi_n(z) = 
        \begin{pmatrix}
            1 & 0 \\ 0 & \kappa_n^{-2}
        \end{pmatrix} Y_n(z) 
        \begin{pmatrix}
            \mathbb{I}(z) & 0 \\ 0 & z^n
        \end{pmatrix} \, , 
        \label{eq:Psi_def}
    \end{align}
    solves the following Riemann--Hilbert problem:
    \begin{enumerate}
        \item $\Psi$ is analytic on  $\mathbb{C}\backslash\mathbb{D}$ with $\mathbb{D} = \mathbb{T} \cup \xi q^{\mathbb{Z}_{>0}'} \cup \xi^{-1} q^{\mathbb{Z}_{<0}'}$.

        \item The boundary values on $\mathbb{T}$ denoted by $\Psi_\pm$ have the relation, $\displaystyle \Psi_+ = \Psi_- 
        \begin{pmatrix}
            1 & 1 \\ 0 & 1 
        \end{pmatrix}$.

        \item The asymptotic behavior is given by
        \begin{subequations}
            \begin{align}
                \Psi(z) \xrightarrow{z \to \infty} & \begin{pmatrix}
            1 & 0 \\ 0 & \kappa_n^{-2}
        \end{pmatrix} \left( I + \sum_{k \ge 1} z^{-k} Y_{n;k}^{(\infty)} \right) 
        \begin{pmatrix}
            z^n \mathbb{I}(z) & 0 \\ 0 & 1
        \end{pmatrix} \, , \\[.5em]
        \Psi(z) \xrightarrow{z \to 0} & \begin{pmatrix}
            1 & 0 \\ 0 & \kappa_n^{-2}
        \end{pmatrix} \left( Y_n(0) + \sum_{k\ge 1} z^{k} Y_{n;k}^{(0)} \right) 
        \begin{pmatrix}
            \mathbb{I}(z) & 0 \\ 0 & z^n
        \end{pmatrix} \, .
            \end{align}
        \end{subequations}
    \end{enumerate}
\end{proposition}
\begin{proof}
    The analyticity of $\Psi$ directly follows from that of $Y$.
    In this case, we also exclude the set of the poles of the weight function $\mathbb{I}$ in addition to the unit circle.
    We obtain the jump matrix as follows,
    \begin{align}
        \Psi_+ & = 
        \begin{pmatrix}
            1 & 0 \\ 0 & \kappa_n^{-2}
        \end{pmatrix} Y_+
        \begin{pmatrix}
            \mathbb{I} & 0 \\ 0 & z^n
        \end{pmatrix} = 
        \begin{pmatrix}
            1 & 0 \\ 0 & \kappa_n^{-2}
        \end{pmatrix} Y_- 
        \begin{pmatrix}
            1 & z^{-n} \mathbb{I} \\ 0 & 1
        \end{pmatrix}
        \begin{pmatrix}
            \mathbb{I} & 0 \\ 0 & z^n
        \end{pmatrix}
        \nonumber \\
        & = 
        \Psi_-
        \begin{pmatrix}
            \mathbb{I}^{-1} & 0 \\ 0 & z^{-n}
        \end{pmatrix}
        \begin{pmatrix}
            1 & z^{-n} \mathbb{I} \\ 0 & 1
        \end{pmatrix}
        \begin{pmatrix}
            \mathbb{I} & 0 \\ 0 & z^n
        \end{pmatrix}
        = \Psi_-
        \begin{pmatrix}
            1 & 1 \\ 0 & 1
        \end{pmatrix} \, .
    \end{align}
\end{proof}

\begin{proposition}
    We have the following Lax pair,
    \begin{align}
    \Psi_{n+1}(z) = U_n(z) \Psi_n(z) \, , \qquad 
    \Psi_n(qz) = T_n(z) \Psi_n(z) \, .
    \end{align}
    where
    \begin{align}
        U_n(z) = U_{n;1} z + U_{n;0} \, , \qquad
        T_n(z) = \frac{T_{n;2} z^2 + T_{n;1} z + T_{n;0}}{1 - z \xi^{-1} q^{\frac{1}{2}}} \, .
    \end{align}        
    In particular, we have
    \begin{subequations}\label{eq:UT_coefficients}
        \begin{gather}
            U_{n;1} = 
            \begin{pmatrix}
                1 & 0 \\ 0 & 0
            \end{pmatrix} \, , \qquad 
            U_{n;0} =
            \begin{pmatrix}
            x_{n} x_{n+1} & - x_{n+1} \\ - (1 - x_{n+1}^2) x_n & 1 - x_{n+1}^2
            \end{pmatrix} \, , \label{eq:U_coefficients} \\
            T_{n;2} = 
        \begin{pmatrix}
            q^{n+1} & 0 \\ 0 & 0
        \end{pmatrix} \, , \qquad 
        T_{n;0} = q^n
        \begin{pmatrix}
            1 - x_n^2 & x_n \\ (1-x_n^2) x_n & x_n^2
        \end{pmatrix} \, . \label{eq:T_coefficients}
        \end{gather}
    \end{subequations}
\end{proposition}

\begin{proof}
    The derivation of the matrix $U_n$ is well established~\cite{Baik2003} (see also \cite{Chouteau:2023,Chouteau:2023PhD}).
    In the domain $\mathbb{C}\backslash\{\mathbb{D}\cup 0\}$, the matrix $\Psi_n$ is invertible.
    Hence, we have
    \begin{align}
        U_n(z) = \Psi_{n+1}(z) \Psi_n(z)^{-1} = 
        \begin{pmatrix}
            1 & 0 \\ 0 & \kappa_{n+1}^{-2}
        \end{pmatrix} Y_{n+1}(z) 
        \begin{pmatrix}
            1 & 0 \\ 0 & z
        \end{pmatrix} Y_n(z)^{-1}
        \begin{pmatrix}
            1 & 0 \\ 0 & \kappa_n^{2}
        \end{pmatrix} \, ,
    \end{align}
    which is analytic on $\mathbb{C}$ since $\Psi_{n+1,+}\Psi_{n,+}^{-1} = \Psi_{n+1,-}\Psi_{n,-}^{-1}$.
    Therefore, $U_n$ is a matrix-valued polynomial in $z$, whose coefficients can be determined by the asymptotic behavior.
    From the behavior of $Y_n$ and $Y_{n+1}$ as $z \to \infty$, we have
    \begin{align}
        U_n(z) = 
        \begin{pmatrix}
            z & 0 \\ 0 & 0    
        \end{pmatrix}(1 + O(z^{-1})) \, , \quad z \to \infty \, ,
    \end{align}
    from which we obtain $U_{n;1}$.
    On the other hand, near $z = 0$, we have
    \begin{align}
        U_n(z) & = 
        \begin{pmatrix}
            1 & 0 \\ 0 & \kappa_{n+1}^{-2}
        \end{pmatrix} Y_{n+1}(0) 
        \begin{pmatrix}
            1 & 0 \\ 0 & 0
        \end{pmatrix} Y_n(0)^{-1}
        \begin{pmatrix}
            1 & 0 \\ 0 & \kappa_n^{2}
        \end{pmatrix} (1 + O(z))
        \nonumber \\
        & = 
        \begin{pmatrix}
            x_{n} x_{n+1} & - x_{n+1} \\ - (1 - x_{n+1}^2) x_n & 1 - x_{n+1}^2
        \end{pmatrix} (1 + O(z)) \, , \quad z \to 0 \, ,
    \end{align}
    which determines $U_{n;0}$.

    For the matrix $T_n$, we have $\Psi_+(qz) \Psi_+(z)^{-1} = \Psi_-(qz) \Psi_-(z)^{-1}$, hence it is analytic on $\mathbb{T}$.
    We compute
    \begin{align}
        T_n(z) & = \Psi_n(qz) \Psi_n(z)^{-1} = 
        \begin{pmatrix}
            1 & 0 \\ 0 & \kappa_n^{-2}
        \end{pmatrix} Y_n(qz) 
        \begin{pmatrix}
            \frac{\mathbb{I}(qz)}{\mathbb{I}(z)} & 0 \\ 0 & q^n
        \end{pmatrix} Y_n(z)^{-1}
        \begin{pmatrix}
            1 & 0 \\ 0 & \kappa_n^{2}
        \end{pmatrix}         
    \end{align}
    where we have
    \begin{align}
        \frac{\mathbb{I}(qz)}{\mathbb{I}(z)} = \frac{1 - \xi q^{\frac{1}{2}} z}{1 - \xi q^{-\frac{1}{2}} / z} =
        \begin{cases}
            - \xi q^{\frac{1}{2}} z (1 + O(z^{-1})) & (z \to \infty) \\
            O(z) & (z \to 0)
        \end{cases}
    \end{align}
    Hence, the matrix $T_n$ may have a pole at $z = \xi q^{-\frac{1}{2}}$.
    We compute the asymptotic behavior as $z \to \infty$,
    \begin{align}
        T_n(z) 
        & =
        \begin{pmatrix}
        -\xi q^{n+\frac{1}{2}} z & 0 \\ 0 & 0
        \end{pmatrix} (1 + O(z^{-1})) \, , \quad z \to \infty \, ,
    \end{align}
    and, as $z \to 0$, we have
    \begin{align}
        T_n(z) & = 
        \begin{pmatrix}
            1 & 0 \\ 0 & \kappa_n^{-2}
        \end{pmatrix} Y_n(0) 
        \begin{pmatrix}
            0 & 0 \\ 0 & q^n
        \end{pmatrix} Y_n(0)^{-1}
        \begin{pmatrix}
            1 & 0 \\ 0 & \kappa_n^{2}
        \end{pmatrix} (1 + O(z)) \nonumber\\ 
        & = q^n
        \begin{pmatrix}
        1-x_n^2 & x_n \\ (1-x_n^2)x_n & x_n^2
        \end{pmatrix} (1 + O(z))
        \, , \quad z \to 0 \, ,
    \end{align}
    from which we determine $T_{n;2}$ and $T_{n;0}$.
\end{proof}

\begin{remark}
    Joshi and Latimer studied the Riemann--Hilbert problem associated with $q$-orthogonal polynomials~\cite{Joshi:2021wno}, where the inner product is defined based on the Jackson integral.
    Although our inner product also allows an infinite series form~\eqref{eq:Jackson_form}, it is not written as a Jackson integral, but rather as a $q$-hypergeometric series~\cite{Koekoek2010}.
    Generalizing Joshi--Latimer's result to such a $q$-hypergeometric type orthogonal polynomial would be an interesting future direction.
\end{remark}

The Lax pairs for the $q$-Painlevé equations have been constructed for $q$-P$_\text{VI}$ by Jimbo and Sakai~\cite{Jimbo1996} and for the degenerate cases by Murata~\cite{Murata2009} based on Sakai's geometric classification~\cite{Sakai2001}. 
Let $Y(z,t)$ a matrix-valued function of size 2.
Applying Jimbo--Sakai and Murata's notation, we consider the following shift equations,
\begin{align}
    Y(qz,t) = A(z,t) Y(z,t) \, , \qquad Y(z,qt) = B(z,t) Y(z,t) \, ,
\end{align}
and the compatibility condition $A(z,qt) B(z,t) = B(qz,t) A(z,t)$ gives us the $q$-Painlevé equations.
Comparing with our notation, we may have the correspondence, $(A(z,q^n),B(z,q^n)) \longleftrightarrow (T_n(z),U_n(z))$. 
For example, the Lax pair for $q$-P$(A_5)^\sharp$ considered in~\cite{Murata2009} is given by
\begin{align}
    A(x,t) = A_0(t) + z A_1(t) + z^2 A_2(t) \, , \qquad B(x,t) = \frac{z + B_0(t)}{z - a q t}
\end{align}
where $A_2 = \operatorname{diag}(\kappa,0)$ and $A_0(t)$ has eigenvalues $(\theta t,0)$ with parameters $(a,\kappa,\theta)$.
In our case, we have $T_{n;2} = \operatorname{diag}(q^{n+1},0)$ and $T_{n;0}$ has eigenvalues $(q^n,0)$.
Although there is such a similarity, an essential difference is that $A$ ($U$, resp.) does not have a pole, whereas $T$ ($B$) does, and hence an explicit relation between two formalisms is not clear at this moment.
We will leave this issue for future work.

In order to fix the remaining matrix coefficient $T_{n;1}$, we prepare the following Lemma.

\begin{lemma}\label{lem:T_inversion}
    The following inversion relation holds,
    \begin{align}
        T_n(z)^{-1} = q^{-n} K_n T_n((qz)^{-1}) K_n \, ,
    \end{align}
    where
    \begin{align}
        K_n = K_n^{-1} = 
        \begin{pmatrix}
            1 & 0 \\ 0 & \kappa_n^{-2}
        \end{pmatrix} Y_n(0) \sigma_3 
        \begin{pmatrix}
            1 & 0 \\ 0 & \kappa_n^2
        \end{pmatrix} = 
        \begin{pmatrix}
            x_n & -1 \\ -(1-x_n^2) & -x_n
        \end{pmatrix}
        \, ,
    \end{align}
    with $\det K_n = -1$.
\end{lemma}
\begin{proof}
    From the identity~\eqref{eq:Y_inversion}, we have $Y_n(z^{-1}) = Y_n(0) \sigma_3 Y_n(z) \sigma_3 z^{-n \sigma_3}$. Then, by the definition~\eqref{eq:Psi_def}, we have 
    \begin{align}
        \Psi_n(z^{-1}) = 
        \begin{pmatrix}
            1 & 0 \\ 0 & \kappa_n^{-2}
        \end{pmatrix} Y_n(z^{-1})
        \begin{pmatrix}
            \mathbb{I}(z^{-1}) & 0 \\ 0 & z^{-n}
        \end{pmatrix} 
        = 
        \begin{pmatrix}
            1 & 0 \\ 0 & \kappa_n^{-2}
        \end{pmatrix} 
        Y_n(0) \sigma_3 Y_n(z) \sigma_3 z^{-n \sigma_3}
        \begin{pmatrix}
            \mathbb{I}(z^{-1}) & 0 \\ 0 & z^{-n}
        \end{pmatrix} \, .
    \end{align}
    By the symmetry $\mathbb{I}(z) = \mathbb{I}(z^{-1})$, we have $\sigma_3 z^{-n \sigma_3} \operatorname{diag}(\mathbb{I}(z^{-1}),z^{-n}) = z^{-n} \operatorname{diag}(\mathbb{I}(z),z^{n}) \sigma_3$, and hence we obtain
    \begin{align}
        \Psi_n(z^{-1}) = z^{-n}
        \begin{pmatrix}
            1 & 0 \\ 0 & \kappa_n^{-2}
        \end{pmatrix} Y_n(0) \sigma_3
        \begin{pmatrix}
            1 & 0 \\ 0 & \kappa_n^{2}
        \end{pmatrix} \Psi_n(z) \sigma_3 = z^{-n} K_n \Psi_n(z) \sigma_3 \, .
    \end{align}
    Recalling $\Psi_n(q^{-1}z) = T_n(q^{-1}z)^{-1} \Psi_n(z)$, we compute
    \begin{align}
        \Psi_n(q z^{-1}) 
        & = q^n z^{-n}
        \begin{pmatrix}
            1 & 0 \\ 0 & \kappa_n^{-2}
        \end{pmatrix} Y_n(0) \sigma_3 
        \begin{pmatrix}
            1 & 0 \\ 0 & \kappa_n^{2}
        \end{pmatrix} \Psi_n(q^{-1}z) \sigma_3 \nonumber \\
        & = q^n z^{-n}
        \begin{pmatrix}
            1 & 0 \\ 0 & \kappa_n^{-2}
        \end{pmatrix} Y_n(0) \sigma_3 
        \begin{pmatrix}
            1 & 0 \\ 0 & \kappa_n^{2}
        \end{pmatrix} T_n(q^{-1}z)^{-1} \Psi_n(z) \sigma_3 \, ,
    \end{align}
    from which we obtain $T_n(z^{-1}) = \Psi_n(q z^{-1}) \Psi_n(z^{-1})^{-1} = q^n K_n T_n(q^{-1}z)^{-1} K_n$ and conclude the proof.
\end{proof}

\begin{theorem}\label{thm:qPV1}
    Set $\mathsf{x}_n = \xi^{\frac{1}{2}} q^{\frac{n}{2}} x_n$.
    The compatibility condition
    \begin{align}
    U_n(qz) T_n(z) = T_{n+1}(z) U_n(z) \label{eq:comp_eq1} \, ,
    \end{align}
    is equivalent to $q$-P$_\text{V}$,
    \begin{align}
    \left( \mathsf{x}_n \mathsf{x}_{n+1} - 1 \right) \left( \mathsf{x}_{n-1} \mathsf{x}_{n} - 1 \right) = \frac{(\mathsf{x}_n - a)(\mathsf{x}_n - a^{-1})(\mathsf{x}_n - b)(\mathsf{x}_n - b^{-1})}{(1 - c q_n  \mathsf{x}_n)(1 - c^{-1} q_n \mathsf{x}_n)} \, , \label{eq:qPV}
    \end{align}
    under the parameter identification,
    \begin{align}
    a = \xi^{\frac{1}{2}} \, , \quad b = - \xi^{\frac{1}{2}}, \quad c = \ii \, , \quad q_n = q_0 \lambda^n \, , \quad q_0 = (-\xi)^{-\frac{1}{2}} \, \quad \lambda = q^{-\frac{1}{2}} \, .
    \end{align}
    Namely, we have
    \begin{align}
        \left( \mathsf{x}_n \mathsf{x}_{n+1} - 1 \right) \left( \mathsf{x}_{n-1} \mathsf{x}_{n} - 1 \right) & = \frac{(\mathsf{x}_n^2 - \xi)(\mathsf{x}_n^2 - \xi^{-1})}{(1 - \xi^{-1} q^{-n} \mathsf{x}_n^2)} \, , \label{eq:x_recurrence}
    \end{align}
    with the initial condition $\mathsf{x}_0 = \xi^{\frac{1}{2}}$.
\end{theorem}

\begin{proof}
    We should determine the matrix $T_n$, where we already determined $T_{n;2}$ and $T_{n;0}$.
    Set
    \begin{align}
        T_{n;1} =
        \begin{pmatrix}
            \alpha_n & \beta_n \\ \gamma_n & \delta_n
        \end{pmatrix} \, .
    \end{align}
    The compatibility condition~\eqref{eq:comp_eq1} is equivalent to the following set of matrix equations,
    \begin{subequations}
        \begin{align}
            q U_{n;1} T_{n;2} & = T_{n+1,2} U_{n;1} \, , \\
            q U_{n;1} T_{n;1} + U_{n;0} T_{n;2} & = T_{n+1;2} U_{n;0} + T_{n+1;1} U_{n;1} \label{eq:compatibility_12} \, , \\
            q U_{n;1} T_{n;0} + U_{n;0} T_{n;1} & = T_{n+1;1} U_{n;0} + T_{n+1;0} U_{n;1} \, , \label{eq:compatibility_13} \\
            U_{n;0} T_{n;0} & = T_{n+1;0} U_{n;0} \, .
        \end{align}
    \end{subequations}
    The coefficient $T_{n;1}$ is involved in the second and the third equations.
    We verify that the other coefficients of $U_n$ and $T_n$ shown in~\eqref{eq:UT_coefficients} are consistent with the first and the fourth equations.
    For the second equation, we compute
    \begin{align}
        & q U_{n;1} T_{n;1} + U_{n;0} T_{n;2} - T_{n+1;2} U_{n;0} - T_{n+1;1} U_{n;1} \nonumber \\
        & = \left(
        \begin{array}{cc}
         q \left(\alpha_n+(1-q) q^n x_n x_{n+1}\right)-\alpha_{n+1} & q \left(\beta_n+q^{n+1} x_{n+1}\right) \\
         -q^{n+1} \left(1-x_{n+1}^2\right) x_n -\gamma_{n+1} & 0 \\
        \end{array}
        \right) \, ,
    \end{align}
    hence we obtain
    \begin{align}
        \beta_n = - q^{n+1} x_{n+1} \, , \qquad 
        \gamma_n = - q^n (1 - x_{n}^2) x_{n-1} \, . 
    \end{align}
    By Lemma~\ref{lem:T_inversion}, the inversion relation of the matrix $T_n$ is equivalent to 
    \begin{subequations}
    \begin{align}
        q^{-n} K_n T_{n;0} K_n T_{n;2} & = 0 \, , \\
        q^{-n} \left( q^{-1} K_n T_{n;1} K_n T_{n;2} + K_n T_{n;0} K_n T_{n;1} \right) & = - q^{\frac{1}{2}} \xi^{-1} \, , \\
        q^{-n} \left( q^{-2} K_n T_{n;2} K_n T_{n;2} + q^{-1} K_n T_{n;1} K_n T_{n;1} + K_n T_{n;0} K_n T_{n;0} \right) & = 1 + \xi^{-2} \, , \label{eq:T_inversion_13}\\
        q^{-n} \left( q^{-2} K_n T_{n;2} K_n T_{n;1} + q^{-1} K_n T_{n;1} K_n T_{n;0} \right) & = - q^{-\frac{1}{2}} \xi^{-1} \, , \\
        q^{-n-2} K_n T_{n;2} K_n T_{n;0} & = 0 \, ,
    \end{align}
    \end{subequations}
    where the right hand side of these equations are scalar matrices.
    The first and the fifth equations are consistent with the coefficients~\eqref{eq:T_coefficients}.
    The second equation reads
    \begin{align}
        0 & = q^{-n} \left( q^{-1} K_n T_{n;1} K_n T_{n;2} + K_n T_{n;0} K_n T_{n;1} \right) + q^{\frac{1}{2}} \xi^{-1} \nonumber \\ & =
        \left(
        \begin{array}{cc}
        x_n \left(\alpha_n x_n-\gamma_n\right)+\left(x_n^2-1\right) \left(\beta_n x_n-\delta_n\right)+q^{\frac{1}{2}}\xi^{-1} & 0 \\
        \left(x_n^2-1\right) \left(x_n \left(\alpha_n-\delta_n\right)+\beta_n \left(x_n^2-1\right)-\gamma_n\right) & \delta_n+q^{\frac{1}{2}}\xi^{-1} \\
        \end{array}
        \right) \, ,
    \end{align}
    from which we determine
    \begin{align}
        \delta_n = - q^{\frac{1}{2}}\xi^{-1} \, .
    \end{align}
    We obtain the same result also from the fourth equation.
    Then, from the compatibility condition~\eqref{eq:compatibility_13}, together with $(\beta_n,\gamma_n,\delta_n)$, we obtain
    \begin{align}
        \alpha_n = \frac{q^n \left(x_n^2-1\right) \left(q x_{n+1}+x_{n-1}\right)}{x_n}-q^{\frac{1}{2}}\xi^{-1} \, .
    \end{align}
    From the second equation of the compatibility condition~\eqref{eq:compatibility_12}, we obtain the recurrence relation for $x_n$,
    \begin{align}
    \left(\xi q^{n + \frac{1}{2}} x_{n-1} x_n x_{n+1} - (x_{n-1} + q x_{n+1}) \right)(1 - x_n^2) = - q^{\frac{1}{2}} \left( \xi (1 - x_n^2) + \xi (1 - q^n) x_n^2 + \frac{1 - q^{-n}}{\xi} \right) x_n \, . \label{eq:x_recurrence1}
    \end{align}
    The same relation is obtained also from the inversion relation \eqref{eq:T_inversion_13}.
    Rewriting this relation in terms of $\mathsf{x}_n = \xi^{\frac{1}{2}} q^{\frac{n}{2}} x_n$, we obtain the result shown in~\eqref{eq:x_recurrence}.
\end{proof}

\begin{corollary}\label{cor:qto1}
    Set $\xi = (1 - q)\eta$.
    Taking the limit $q \to 1$, we obtain d-P$_\text{II}$,
    \begin{align}
     (x_{n-1} + x_{n+1}) (1 - x_n^2) + \frac{n}{\eta} x_n = 0 \, .
    \end{align}
\end{corollary}
\begin{proof}
    It directly follows from the relation~\eqref{eq:x_recurrence1}.
\end{proof}

This result has a natural interpretation as a $q$-deformation of the Tracy--Widom formula for the Fredholm determinant of the Airy kernel~\cite{Tracy:1992rf}.
Let $K_\text{Ai}$ the integral operator associated with the Airy kernel. 
The Fredholm determinant for the interval $[s,\infty)$ is given by
\begin{align}
    \tau(s) := \det(1 - K_\text{Ai})_{L^2([s,\infty))} = \exp \left( - \int_s^\infty (x - s) \mathsf{q}(x)^2 \dd{x} \right) \, ,
\end{align}
where $\mathsf{q}$ solves the Painlevé II equation (P$_\text{II}$), $\mathsf{q}'' = x \mathsf{q} + 2 \mathsf{q}^3$, with the boundary condition, $\mathsf{q}(x) \xrightarrow{x \to \infty} \operatorname{Ai}(x)$ (Hastings--McLeod solution to P$_\text{II}$~\cite{Hastings1980}).

Let $\nabla$ a difference operator: $\nabla f_n = f_{n+1/2} - f_{n-1/2}$.
By Lemma~\ref{lemma:Zx_relation}, we have
\begin{align}
    \nabla^2 \log Z_n = \log (1 - x_n^2) = - x_n^2 + O(x_n^4) \, ,
\end{align}
which is an analogue of the relation, $(\log \tau)'' = - \mathsf{q}^2$. 
Hence, the Fredholm determinant is interpreted as the corresponding $\tau$-function in this context.
Moreover, again by Lemma~\ref{lemma:Zx_relation}, we have the asymptotic behavior, $|\mathsf{x}_n| \ll \xi^{\frac{1}{2}} q^{\frac{n}{2}} \ll 1$, in the limit $n \to \infty$.
In this situation, we may omit the non-linear terms in the recurrence relation to have a linear relation,
\begin{align}
 \mathsf{x}_{n-1} + \mathsf{x}_{n+1} = \left( \frac{1 - q^{-n}}{\xi} + \xi \right) \mathsf{x}_n \, .
\end{align}
This agrees with the recurrence relation for the Hahn--Exton $q$-Bessel function~\eqref{eq:q-Bessel_rec_rel3}, implying the asymptotic behavior, $\mathsf{x}_n \xrightarrow{n \to \infty} J_{-n}^{(3)}(2\xi;q)$.
From this point of view, our solution to $q$-P$_\text{V}$ is interpreted as a $q$-deformed version of the Hastings--McLeod solution.
Such a connection to the continuous case was similarly discussed in the limit $q \to 1$~\cite{Borodin:2003Duke,Chouteau:2023}.

\subsection{Weight function $\check{\mathbb{I}}$}\label{sec:Icheck}

Set $\check{\mathbb{I}}(z) = \check{\mathbb{I}}(z;\xi)$.
Let $\check{Y}_n$ be the matrix-valued function defined for the weight function $\check{\mathbb{I}}$ instead of the previous one $\mathbb{I}$.

\begin{proposition}    
    The matrix-valued function,
    \begin{align}
        \Phi(z) = \Phi_n(z) = 
        \begin{pmatrix}
            1 & 0 \\ 0 & \check{\kappa}_n^{-2}
        \end{pmatrix} \check{Y}_n(z) 
        \begin{pmatrix}
            1 & 0 \\ 0 & z^n \check{\mathbb{I}}(z)^{-1}
        \end{pmatrix} \, , 
        \label{eq:Psi2_def}
    \end{align}
    solves the following Riemann--Hilbert problem:
    \begin{enumerate}
        \item $\Phi$ is analytic on $\mathbb{C}\backslash\check{\mathbb{D}}$ with $\check{\mathbb{D}} = \mathbb{T} \cup (- \xi) q^{\mathbb{Z}_{>0}'} \cup (- \xi^{-1}) q^{\mathbb{Z}_{<0}'}$.

        \item The boundary values on $\mathbb{T}$ denoted by $\Phi_\pm$ have the relation, $\displaystyle \Phi_+ = \Phi_- 
        \begin{pmatrix}
            1 & 1 \\ 0 & 1 
        \end{pmatrix}$.

        \item The asymptotic behavior is given by
        \begin{subequations}
            \begin{align}
                \Phi(z) \xrightarrow{z \to \infty} & \begin{pmatrix}
            1 & 0 \\ 0 & \check{\kappa}_n^{-2}
        \end{pmatrix} 
        \begin{pmatrix}
            z^n & 0 \\ 0 & \check{\mathbb{I}}(z)^{-1}
        \end{pmatrix} (1 + O(z^{-1})) \, , \\[.5em]
        \Phi(z) \xrightarrow{z \to 0} & 
        \begin{pmatrix}
            1 & 0 \\ 0 & \check{\kappa}_n^{-2}
        \end{pmatrix} \check{Y}_n(0) 
        \begin{pmatrix}
            1 & 0 \\ 0 & z^n \check{\mathbb{I}}(z)^{-1}
        \end{pmatrix} (1 + O(z)) \, ,
            \end{align}
        \end{subequations}
        where $\displaystyle \check{Y}_n(0) = 
        \begin{pmatrix}
        y_n & \check{\kappa}_n^{-2} \\ - \check{\kappa}_{n-1}^2 & y_n
        \end{pmatrix}$.
    \end{enumerate}
\end{proposition}

Compared with the previous case, the weight function dependence is different in the matrix-valued function.
In fact, $\check{\mathbb{I}}(z;\xi) = \mathbb{I}(z;-\xi)^{-1}$.
Hence, in this case, we have
\begin{align}
    \frac{\check{\mathbb{I}}(z)}{\check{\mathbb{I}}(qz)} = \frac{1 + \xi q^{\frac{1}{2}} z}{1 + \xi q^{-\frac{1}{2}} / z} =
    \begin{cases}
        \xi q^{\frac{1}{2}} z (1+O(z^{-1})) & (z \to \infty) \\
        O(z) & (z \to 0)
    \end{cases}
    \label{eq:I_ratio2}
\end{align}
This behavior plays an essential role in the Riemann--Hilbert problem as we discuss below.

\begin{proposition}
    We have the following Lax pair,
    \begin{align}
    \Phi_{n+1}(z) = U_n(z) \Phi_n(z) \, , \qquad 
    \Phi_n(qz) = T_n(z) \Phi_n(z) \, .
    \end{align}
    where
    \begin{align}
        U_n(z) = U_{n;1} z + U_{n;0} \, , \qquad
        T_n(z) = \frac{T_{n;2} z^2 + T_{n;1} z + T_{n;0}}{1 + z \xi^{-1} q^{\frac{1}{2}}} \, .
    \end{align}        
    We have the same matrix $U_{n}$ as before \eqref{eq:U_coefficients} under the replacement $x_n \mapsto y_n$, while the coefficients of the matrix $T_n$ are given by 
    \begin{align}
        T_{n;2} = 
        \begin{pmatrix}
            0 & 0 \\ 0 & q
        \end{pmatrix} \, , \qquad 
        T_{n;0} = 
        \begin{pmatrix}
            y_n^2 & - y_n \\ -(1-y_n^2) y_n & 1 - y_n^2
        \end{pmatrix} \, .
    \end{align}
\end{proposition}
\begin{proof}
    The matrix $U_n$ is completely the same as before.
    For the matrix $T_n$, we have
    \begin{align}
        T_n(z) & = \Phi_n(qz) \Phi_n(z)^{-1} = 
        \begin{pmatrix}
            1 & 0 \\ 0 & \check{\kappa}_n^{-2}
        \end{pmatrix} \check{Y}_n(qz) 
        \begin{pmatrix}
            1 & 0 \\ 0 & q^n \frac{\check{\mathbb{I}}(z)}{\check{\mathbb{I}}(qz)}
        \end{pmatrix} \check{Y}_n(z)^{-1}
        \begin{pmatrix}
            1 & 0 \\ 0 & \check{\kappa}_n^{2}
        \end{pmatrix} \, .
    \end{align}
    Hence, from the $q$-shift behavior of the weight function given in \eqref{eq:I_ratio2}, we identify the pole of the matrix $T_n$ at $z = -\xi q^{-\frac{1}{2}}$.
    The asymptotic behavior is given at $z \to \infty$ by
    \begin{align}
        T_n(z) & = 
        \begin{pmatrix}
            0 & 0 \\ 0 & \xi q^{\frac{1}{2}} z
        \end{pmatrix} (1 + O(z^{-1})) \, ,
    \end{align}
    and, at $z \to 0$, we have
    \begin{align}
        T_n(z) & = 
        \begin{pmatrix}
            1 & 0 \\ 0 & \check{\kappa}_n^{-2}
        \end{pmatrix} \check{Y}_n(0) 
        \begin{pmatrix}
            1 & 0 \\ 0 & 0
        \end{pmatrix} \check{Y}_n(0)^{-1}
        \begin{pmatrix}
            1 & 0 \\ 0 & \check{\kappa}_n^{2}
        \end{pmatrix} (1 + O(z)) \nonumber \\
        & = 
        \begin{pmatrix}
            y_n^2 & - y_n \\ -(1-y_n^2) y_n & 1-y_n^2
        \end{pmatrix} (1 + O(z)) \, ,
    \end{align}
    from which we determine $T_{n;2}$ and $T_{n;0}$.
\end{proof}

\begin{theorem}\label{thm:qPV2}
    Set $\mathsf{y}_n = (-\xi)^{\frac{1}{2}} q^{-\frac{n}{2}} y_n$.
    The compatibility condition
    \begin{align}
    U_n(qz) T_n(z) = T_{n+1}(z) U_n(z) \, ,
    \end{align}
    is equivalent to the recurrence relation,
    \begin{align}
        \left( \mathsf{y}_n \mathsf{y}_{n+1} - 1 \right) \left( \mathsf{y}_{n-1} \mathsf{y}_{n} - 1 \right) 
        & = \frac{(\mathsf{y}_n^2 + \xi)(\mathsf{y}_n^2 + \xi^{-1})}{(1 + \xi^{-1} q^{n} \mathsf{y}_n^2)} \, ,
    \end{align}
    with the initial condition $\mathsf{y}_0 = (-\xi)^{\frac{1}{2}}$, which is identified with $q$-P$_\text{V}$~\eqref{eq:qPV} under the parameter identification,
    \begin{align}
    a = (-\xi)^{\frac{1}{2}} \, , \quad b = - (-\xi)^{\frac{1}{2}}, \quad c = \ii \, , \quad q_n = q_0 \lambda^n \, , \quad q_0 = \xi^{-\frac{1}{2}} \, \quad \lambda = q^{\frac{1}{2}} \, .
    \end{align}
\end{theorem}
\begin{proof}
    The same proof is applied to this case as before.
    Compared with the previous case, we obtain the expressions under the replacement, $\xi \mapsto - \xi$.
\end{proof}

In the limit $n \to \infty$, similar to the previous case, we have the asymptotic behavior, $|\mathsf{y}_n| \ll 1$.
Hence, the non-linear recurrence relation is reduced to the linear relation in this limit,
\begin{align}
 \mathsf{y}_{n-1} + \mathsf{y}_{n+1} = - \left( \frac{1 - q^{n}}{\xi} + \xi \right) \mathsf{y}_n \, , 
\end{align}
which is the recurrence relation for the Hahn--Exton $q$-Bessel function, $\mathsf{y}_n \xrightarrow{n \to \infty} J_n^{(3)}(-2\xi;q)$.

\begin{corollary}
    Set $\xi = (1 - q)\eta$.
    In the limit $q\to 1$, we obtain d-P$_\text{II}$,
    \begin{align}
        (y_{n-1} + y_{n+1})(1 - y_n^2) + \frac{n}{\eta} y_n = 0 \, .
    \end{align}
\end{corollary}

\appendix

\section{$q$-functions}\label{sec:q-functions}

\subsubsection*{MacMahon function}
Let $|q| < 1$.
We define a modified MacMahon function,
\begin{align}
    M(\xi;q) = \prod_{n \ge 1} \frac{1}{(1 - \xi^2 q^n)^n} \, . \label{eq:MacMahon_fn}
\end{align}

\begin{proposition}\label{prop:MacMahon_PE}
    We have
    \begin{align}
    M(\xi;q) = \exp \left( \sum_{n \ge 1} \frac{\xi^{2n}}{n(q^{\frac{n}{2}}-q^{-\frac{n}{2}})^2} \right) \, . 
\end{align}
\end{proposition}
\begin{proof}
    We may use the series expansion of $\log (1 - z) = - \sum_{n \ge 1} \frac{z^n}{n}$. For $|q|<1$, we have
    \begin{align}
        \prod_{n \ge 1} \frac{1}{(1 - \xi^2 q^n)^n} 
        & = \exp \left( - \sum_{n \ge 1} n \log (1 - \xi^2 q^n) \right) \nonumber \\
        & = \exp \left( \sum_{n \ge 1} \sum_{m \ge 1} n \frac{\xi^{2m} q^{nm}}{m} \right) \nonumber \\
        & = \exp \left( \sum_{m \ge 1} \frac{\xi^{2m}}{m} \frac{q^m}{(1-q^m)^2} \right) \, ,
    \end{align}
    which agrees with the desired expression.
\end{proof}

Let $p(n)$ be the number of plane partitions of size $n$.
The generating function of $p(n)$ is known to be the MacMahon function,
\begin{align}
    M(1;q) = \prod_{n \ge 1} \frac{1}{(1 - q^n)^n} = \sum_{n = 0}^\infty p(n) q^n \, .
\end{align}

\subsubsection*{$q$-shifted factorial}
We define the $q$-shifted factorial,
\begin{align}
    (x;q)_n := \prod_{k=0}^{n-1} (1 - x q^k) \, .
\end{align}
The limit $\lim_{n \to \infty} (x;q)_n = (x;q)_\infty$ converges for $|q|<1$.
We write
\begin{align}
    (x_1,x_2,\ldots,x_k;q)_n = (x_1;q)_n (x_2;q)_n \cdots (x_k;q)_n \, .
\end{align}

\subsubsection*{Basic hypergeometric series}

We define the basic hypergeometric series,
\begin{align}
    & {}_r\phi_s(a_1,\ldots,a_r;b_1,\ldots,b_s;q,x) 
    = \sum_{n = 0}^\infty \frac{(a_1,\ldots,a_r;q)_n}{(q;b_1,\ldots,b_s;q)_n} \left((-1)^n q^{n \choose 2}\right)^{1+s-r} x^n \, .
\end{align}

\section{$q$-Bessel functions}\label{sec:q-Bessel}

We summarize the definitions and the properties of the $q$-Bessel functions.
We follow the notations of~\cite{Gasper2004}.

\subsubsection*{$q$-Bessel functions}

Let $q \in [0,1)$.
There are three types of $q$-Bessel function,
\begin{subequations}
    \begin{align}
        J_\nu^{(1)}(x;q) & = \frac{(q^{\nu+1};q)_\infty}{(q;q)_\infty} \left(\frac{x}{2}\right)^\nu {}_2\phi_1(0,0;q^{\nu+1};q,-x^2/4) \, , \\
        J_\nu^{(2)}(x;q) & = \frac{(q^{\nu+1};q)_\infty}{(q;q)_\infty} \left(\frac{x}{2}\right)^\nu {}_0\phi_1(-;q^{\nu+1};q,-x^2 q^{\nu+1}/4) \, , \\
        J_\nu^{(3)}(x;q) & = \frac{(q^{\nu+1};q)_\infty}{(q;q)_\infty} \left(\frac{x}{2}\right)^\nu {}_1\phi_1(0;q^{\nu+1};q,qx^2/4) \, , \label{eq:qBessel_J3}
    \end{align}
\end{subequations}
which are reduced to the ordinary Bessel function in the limit $q \to 1$,
\begin{align}
    \lim_{q \to 1} J_\nu^{(j)}((1-q)x;q) = J_\nu(x) \, , \qquad j = 1, 2, 3 \, .
\end{align}
We remark the relation,
\begin{align}
    J_\nu^{(2)}(x;q) = \left(-\frac{x^2}{4};q\right)_\infty J_\nu^{(1)}(x;q) \, , \qquad |x| < 2 \, .
\end{align}

\subsubsection*{Modified $q$-Bessel functions}

We define the modified $q$-Bessel functions,
\begin{align}
    I_\nu^{(j)}(x;q) = \ee^{-\ii\pi\nu/2} J_\nu^{(j)}(x \ee^{\ii\pi/2};q) \, .
\end{align}
There is another expression in terms of the basic hypergeometric series~\cite{Ismail2018},
\begin{subequations}\label{eq:hypergeometric_I}
\begin{align}
    I_\nu^{(1)}(2x;q) & = \frac{x^\nu}{(x^2,q;q)_\infty} {}_1\phi_1(x^2;0;q,q^{\nu+1}) \, , \\
    I_\nu^{(2)}(2x;q) & = \frac{x^\nu}{(q;q)_\infty} {}_1\phi_1(x^2;0;q,q^{\nu+1}) \, .
\end{align}
\end{subequations}

\subsubsection*{Generating functions}

We have the generating functions of the $q$-Bessel functions,
\begin{subequations}
\begin{align}
    \left(\frac{xz}{2},-\frac{x}{2z};q\right)_\infty^{-1} & = \sum_{n \in \mathbb{Z}} z^n J_n^{(1)}(x;q) \, , \\    
    \frac{(qx/2z;q)_\infty}{(xz/2;q)_\infty} & = \sum_{n \in \mathbb{Z}} z^n J_n^{(3)}(x;q) \, .
\end{align}
\end{subequations}
For the modified $q$-Bessel functions, we have,
\begin{subequations}\label{eq:Bessel_gen_fn_mod}
\begin{align}
    \left(\frac{xz}{2},\frac{x}{2z};q\right)_\infty^{-1} & = \sum_{n \in \mathbb{Z}} z^n I_n^{(1)}(x;q) \, , \\ 
    \left(-\frac{zx}{2},-\frac{qx}{2z};q\right)_\infty & = \sum_{n \in \mathbb{Z}} z^n q^{n \choose 2} I_n^{(2)}(x;q) \, .
\end{align}
\end{subequations}

\subsubsection*{Recurrence relations}

The $q$-Bessel functions obey the following recurrence relations,
\begin{subequations}\label{eq:q-Bessel_rec_rel}
    \begin{align}
        q^{\nu} J_{\nu+1}^{(j)}(x;q) & = \frac{2}{x}(1 - q^\nu) J^{(j)}_\nu(x;q) - J^{(j)}_{\nu-1}(x;q) \, , \qquad j = 1, 2 \, , \\
        J^{(3)}_{\nu+1}(x;q) & = \left( \frac{2}{x} (1-q^\nu) + \frac{x}{2} \right) J^{(3)}_\nu(x;q) - J^{(3)}_{\nu-1}(x;q) \, , \label{eq:q-Bessel_rec_rel3}
    \end{align}
\end{subequations}
and
\begin{align}
    q^{\nu} I_{\nu+1}^{(j)}(x;q) & = - \frac{2}{x}(1 - q^\nu) I^{(j)}_\nu(x;q) + I^{(j)}_{\nu-1}(x;q) \, , \qquad j = 1, 2 \, .
\end{align}
We have the following $q$-shift relations,
\begin{subequations}
    \begin{align}
        J_\nu^{(1)}(xq^{\frac{1}{2}};q) & = q^{\pm \frac{\nu}{2}} \left( J_\nu^{(1)}(x;q) \pm \frac{x}{2} J_{\nu\pm 1}^{(1)}(x;q) \right) \, , \\
        I_\nu^{(1)}(xq^{\frac{1}{2}};q) & = q^{\pm \frac{\nu}{2}} \left( I_\nu^{(1)}(x;q) - \frac{x}{2} I_{\nu\pm 1}^{(1)}(x;q) \right) \, ,
    \end{align}
\end{subequations}
and hence we have
\begin{subequations}
    \begin{align}
        J_\nu^{(1)}(xq;q) & = \left( q^{\pm \nu} - \frac{x^2}{4} q^{-\frac{1}{2}} \right) J_\nu^{(1)}(x) \pm \frac{x}{2} \left( q^{\pm \nu} + q^{-\frac{1}{2}} \right) J_{\nu \pm 1}^{(1)}(x;q) \, , \\
        I_\nu^{(1)}(xq;q) & = \left( q^{\pm \nu} + \frac{x^2}{4} q^{-\frac{1}{2}} \right) I_\nu^{(1)}(x) - \frac{x}{2} \left( q^{\pm \nu} + q^{-\frac{1}{2}} \right) I_{\nu \pm 1}^{(1)}(x;q) \, .
    \end{align}
\end{subequations}

\bibliographystyle{ytamsalpha}
\bibliography{ref}

\end{document}